\documentclass[11pt]{amsart}

\usepackage{amssymb,latexsym}

\usepackage{graphicx,epsfig}

\usepackage[dvipsnames]{xcolor}

\def\cal{\mathcal}

\def\Bbb{\mathbb}

\def\R{{\Bbb R}}	

\def\ext{{\cal E}} 

\def\cP{{\cal P}}
\def\I{{\cal I}}
\def\cR{{\cal R}}
\def\cT{{\cal T}}
\def\Landau{{\cal O}}

\def\de{{\delta}}

\def \supp {\text{\rm supp\,}}
\def \graph {\text{\rm graph\,}}

\def\trans{{\,}^t}

\newtheorem{thmnr}{Theorem}[section]
\newtheorem{lemnr}[thmnr]{Lemma}

\newtheorem{remnr}[thmnr]{Remark}

\newtheorem{definr}{Definition}[section]

\begin{document}

\title[A Fourier restriction theorem for a perturbed hyperbolic paraboloid]{A Fourier
restriction
theorem for a perturbed hyperbolic paraboloid\\
}

\author[S. Buschenhenke]{Stefan Buschenhenke}
\address{S. Buschenhenke:  Mathematisches Seminar, C.A.-Universit\"at Kiel,
Ludewig-Meyn-Stra\ss{}e 4, D-24118 Kiel, Germany}
\email{{\tt buschenhenke@math.uni-kiel.de}}
\urladdr{http://www.math.uni-kiel.de/analysis/de/buschenhenke}

\author[D. M\"uller]{Detlef M\"uller}
\address{D. M\"uller: Mathematisches Seminar, C.A.-Universit\"at Kiel,
Ludewig-Meyn-Stra\ss{}e 4, D-24118 Kiel, Germany}
\email{{\tt mueller@math.uni-kiel.de}}
\urladdr{http://www.math.uni-kiel.de/analysis/de/mueller}

\author[A. Vargas]{Ana Vargas}
\address{A. Vargas: Departmento de Mathem\'aticas, Universidad Aut\'onoma de  Madrid,
28049
Madrid,
Spain
}
\email{{\tt ana.vargas@uam.es}}
\urladdr{{http://matematicas.uam.es/~AFA/}}

\thanks{The first author was partially supported by the ERC grant 307617.\\
The first two authors were partially supported by the DFG grants MU 761/ 11-1 and MU 761/
11-2.\\
The third author was partially supported by grants MTM2013--40945 (MINECO) and
MTM2016-76566-P (Ministerio de Ciencia, Innovaci$\acute{\text{o}}$n y Universidades), Spain.}

\begin{abstract} In contrast to elliptic surfaces, the Fourier restriction problem for
hypersurfaces of non-vanishing  Gaussian curvature  which admit principal  curvatures of
opposite signs is still  hardly understood. In fact, even for 2-surfaces, the only case of
a hyperbolic surface for which Fourier restriction estimates  could be established that
are
analogous to the ones known for elliptic surfaces is the hyperbolic paraboloid  or
``saddle'' $z=xy.$
 The bilinear method gave here sharp results for $p>10/3$ (\cite{lee05}, \cite{v05},
 \cite{Sto}), and this result was recently improved to $p>3.25$ \cite{chl17} \cite{k17}.
  This paper  aims to be a first step in extending  those results to more general
  hyperbolic  surfaces. We consider a specific  cubic perturbation of  the saddle and
  obtain the sharp result, up to the end-point, for $p>10/3.$ In the application of the
  bilinear method, we show that the behavior at small scales in our surface is drastically
  different from the saddle. Indeed, as it turns out, in some regimes the perturbation
  term assumes a dominant role, which necessitates  the introduction of a number of new
  techniques that should also  be useful for the study  of more general hyperbolic
  surfaces.
This specific perturbation has turned out to be of fundamental importance also to the understanding of more
general
  classes of perturbations.
\end{abstract}

\maketitle


\tableofcontents

\thispagestyle{empty}

\setcounter{equation}{0}
\section{Introduction}\label{intro}

E. M. Stein proposed in the seventies the problem of restriction of the Fourier
transform to hypersurfaces. Given a smooth hypersurface  $S$  in $\R^n$ with surface
measure $d\sigma_S,$  he asked for the range  of exponents
$\tilde p$ and $\tilde q$ for which the estimate
\begin{align}
\bigg(\int_S|\widehat{f}|^{\tilde q}\,d\sigma_S\bigg)^{1/\tilde q}\le C\|f\|_{L^{\tilde
p}(\R^n)}
\end{align}
holds  true for   every Schwartz function   $f\in\cal S(\R^n),$ with a constant $C$
independent of $f.$
The sharp range in dimension  $n=2$ for curves with non-vanishing curvature was determined
through work by  C. Fefferman, E. M. Stein and A. Zygmund \cite{F1}, \cite{Z}. In higher
dimension, the sharp $L^{\tilde p}-L^2$ result  for hypersurfaces with
non-vanishing Gaussian curvature was obtained by E. M. Stein and P. A. Tomas \cite{To},
\cite{St1} (see also Strichartz \cite{Str}). Some more general classes of surfaces were
treated by A. Greenleaf \cite{Gr}. Many years later, general finite type surfaces in
$\R^3$
(without  assumptions on the curvature)  have been considered in work by I. Ikromov, M.
Kempe and D. M\"uller  \cite{ikm}  and Ikromov and M\"uller \cite{IM-uniform}, \cite{IM},
and the sharp range of Stein-Tomas type    $L^{\tilde p}-L^2$  restriction
estimates has been  determined  for a large class  of smooth, finite-type hypersurfaces,
including all analytic hypersurfaces.

The question about  general  $L^{\tilde p}-L^{\tilde q}$ restriction
estimates is nevertheless still wide open.  Fundamental progress has been made since the
nineties,  with major new ideas  introduced by J. Bourgain  (see for instance
\cite{Bo1}, \cite{Bo2}) and T. Wolff (\cite{W1}),  which led to a better understanding of
the case of  non-vanishing Gaussian  curvature. These ideas and methods were further
developed by A. Moyua, A. Vargas, L. Vega and T. Tao (\cite{MVV1}, \cite{MVV2}
\cite{TVV}),
who established  the so-called bilinear approach (which   had been anticipated in the work
of C. Fefferman \cite{F1} and  had  implicitly been  present in the work of J. Bourgain
\cite{Bo3}) for hypersurfaces with non-vanishing Gaussian curvature for which all
principal
curvatures  have the same sign. The  same method was applied to  the  light cone by
Tao-Vargas (see \cite{TV1}, \cite{TV2}).  A culmination  of the application of the
bilinear  method to such types of  surfaces  was reached in work by T. Tao \cite{T2} (for
positive principal  curvatures), and T. Wolff \cite{W2}  and T. Tao \cite{T4} (for the
light cone). In particular, in these last  two papers the sharp  linear restriction
estimates for the  light  cone in  $\R^4$  were obtained.

In the last years,  J. Bourgain and L. Guth \cite{BoG} made further important  progress on
the case of non-vanishing curvature by making use  also of  multilinear  restriction
estimates due to  J. Bennett, A. Carbery and T. Tao \cite{BCT}. Later L. Guth \cite{Gu16},
\cite{Gu17} improved  these results by using the polynomial partitioning method.

For the case of non-vanishing curvature but principal curvatures of different signs, the
bilinear method was applied independently  by S. Lee \cite{lee05}, and A. Vargas
\cite{v05}, to a specific surface, the hyperbolic paraboloid (or "saddle"). They obtained
a
result which is analogous to Tao's theorem \cite{T2} except for the end-point. B. Stovall
\cite{Sto} recently proved the end-point case. Also, C. H. Cho and J. Lee \cite{chl17} and
J. Kim \cite{k17} improved the range. Perhaps surprisingly at a first thought, their
methods did not give the desired result for any other surface with negative Gaussian
curvature. A key element in their proofs is the fact that the hyperbolic paraboloid is
invariant under certain anisotropic dilations, and no other surface satisfies this same
invariance.
\medskip

Our aim in this article is to provide some  first steps towards gaining an understanding
of  Fourier restriction   to more general hyperbolic surfaces,  by   generalizing Lee's and
Vargas' result on the saddle to certain model surface $S,$ namely the graph of the
function
$\phi(x,y):=xy+\frac{1}{3}y^3$ over a given  small neighborhood of  the origin. Observing
that our specific  $\phi$ is homogeneous under the parabolic scalings $(x,y)\mapsto (r^2
x,
ry), \, r>0,$ we may here assume as well that
\begin{align}\label{surface}
S:=\{(x,y,xy+y^3/3):(x,y)\in Q:=I\times I\},
\end{align}
where $I:=[-1,1].$

\smallskip

\noindent{\bf Remark.} We do concentrate here on the seemingly particular case of a perturbation of the form $y^3/3$ in
order
to not to have to deal with additional technical issues which come up when studying more general perturbations,
say, of the form
$f(y),$ but should like to mention that the understanding of this particular perturbation is indeed the key to
understanding more
general ones depending on $x,$ or $y,$ only. Indeed, in the follow-up preprint \cite{bmv19}, we show how finite
type perturbations  $f(y)$ can be essentially reduced to the special case studied here,  and in further article
we shall also study flat perturbations $f(y).$

\smallskip

As usual, it will be more convenient to use duality and work in the adjoint setting.  If
$\cR$ denotes the Fourier restriction operator $g\mapsto \cR g:=\hat g|_S$ to the surface
$S,$ its adjoint operator $\cR^*$ is given by $\cR^*f(\xi)=\ext f(-\xi),$ where
 $\ext$ denotes the ``Fourier extension'' operator given by
\begin{align}\label{defop}
	\ext f(\xi):=\widehat{f\,d\sigma_S}(\xi)= \int_S f(x)e^{-i\xi\cdot x}\,d\sigma_S(x),
\end{align}
with  $f\in L^q(S,\sigma_S).$ The restriction problem is therefore
equivalent to the question of finding the appropriate range of exponents for which the
estimate
$$
\|\mathcal E f\|_{L^r(\R^3)}\le C\|f\|_{L^q(S,d\sigma_S)}
$$
holds true with a constant $C$ independent of the function   $f\in L^q(S,d\sigma_s).$

By identifying a point $(x,y)\in Q$ with the corresponding point $(x,y,\phi(x,y))$ on $S,$
we may regard our Fourier extension operator $\ext$  as well as an operator mapping
functions on $Q$ to functions on $\R^3,$ which in  terms of our phase function
$\phi(x,y)=xy+y^3/3$ can be expressed  more explicitly  in the form
$$
\ext f(\xi)=\int_Q f(x,y) e^{-i(\xi_1 x+\xi_2 y+\xi_3\phi(x,y))} \eta(x,y) \, dx dy,
$$
if $\xi= (\xi_1,\xi_2,\xi_3)\in \R^3,$ with a smooth density $\eta.$

Our main result is the following

\begin{thmnr}\label{mainresult}
	Assume that $r>10/3$ and  $1/q'>2/r,$ and let $\ext$ denote  the Fourier extension
operator associated to the graph $S$ of the above phase function $\phi$. Then
	\begin{align*}
		\|\ext f\|_{L^r(\R^3)} \leq C_{r,q} \|f\|_{L^q(Q)}
	\end{align*}
	for all $f\in L^q(Q)$.
\end{thmnr}

In the remaining  part of this section,  we shall describe our strategy of proof, and some
of the obstacles that have to be dealt with.

We are going to follow the bilinear approach, which is based on  bilinear estimates of the
form
\begin{align}\label{bil1}
 \|\ext_{U_1}(f_1)\,\ext_{U_2}(f_2)\|_p \leq C(U_1,U_2) \|f_1\|_2\|f_2\|_2.
\end{align}
 Here,  $\ext_{U_1}$ and $\ext_{U_2}$ are the Fourier extension operators associated to
 patches of sub-surfaces $S_i:=\graph \phi|_{U_i}\subset S,\ i=1,2,$  with $U_i\subset Q.$
 What is crucial for obtaining  useful bilinear estimates is that the two  patches of
 surface $S_1$ and $S_2$ satisfy certain {\it transversality conditions,} which  are
 stronger than  just assuming that $S_1$ and $S_2$ are transversal as hypersurfaces (i.e.,
 that all normals to $S_1$ are transversal to all normals to $S_2$). Indeed, what  is
 needed in addition is the following:
 \smallskip

Translate the patches $S_1$ and $S_2$ so that they intersect in a smooth curve. Then the
normals to, say,  $S_1$  for base points  varying along this intersection curve form a
cone
$\Gamma_1.$  What is needed in the bilinear argument is that all the normals  to the
surfaces $S_2$ pass  transversally  through this cone $\Gamma_1,$  and that the analogous
condition holds true with the roles  of $S_1$ and $S_2$ interchanged.

 For more details  on this condition, we refer to the corresponding literature dealing
 with
 bilinear estimates,  for instance \cite{lee05}, \cite{v05}, \cite{lv10}, or \cite{be16}.
In particular, according to Theorem 1.1 in \cite{lee05}, transversality is achieved if the
modulus  of the
following quantity
\begin{align}\label{transs}
\Gamma^\phi_{z}(z_1,z_2,z_1',z_2'):=	\left\langle
(H\phi)^{-1}(z)(\nabla\phi(z_2)-\nabla\phi(z_1)),\nabla\phi(z_2')-\nabla\phi(z_1')\right\rangle
\end{align}
is bounded from below for any $z_i=(x_i,y_i),\, z_i'\in U_i$, $i=1,2$, $z=(x,y)\in U_1\cup
U_2,$ $H\phi$
denoting the Hessian of $\phi$.
If this inequality holds, then we have \eqref{bil1} for $p>5/3,$ with a constant $C$ that
depends only
on an lower bound of (the modulus of) \eqref{transs}, and  on  upper bounds for the
derivatives of $\phi.$  If $U_1$ and
$U_2$ are sufficiently small (with sizes depending on upper bounds of the first and second
order derivatives of $\phi$ and a lower bound for the determinant of $H\phi$) this
condition  reduces to the estimate
\begin{align}
|\Gamma^\phi_{z}(z_1,z_2)|\geq c,
\end{align}
for $z_i=(x_i,y_i)\in U_i$, $i=1,2$, $z=(x,y)\in U_1\cup U_2$, where
\begin{align}\label{trans}
\Gamma^\phi_{z}(z_1,z_2):=	\left\langle
(H\phi)^{-1}(z)(\nabla\phi(z_2)-\nabla\phi(z_1)),\nabla\phi(z_2)-\nabla\phi(z_1)\right\rangle.
\end{align}

It is easy to check that for $\phi(x,y)=xy+y^3/3$, we have
\begin{eqnarray} \label{gammaz}
\Gamma^\phi_{z}(z_1,z_2)
	 &=& 2(y_2-y_1)[x_2-x_1+(y_1+y_2-y)(y_2-y_1)]\\
	 &=:& 2(y_2-y_1)\tau_{z}(z_1,z_2). \label{TV1}
\end{eqnarray}

This should be compared to the case of the non-perturbed hyperbolic paraboloid  (the
``saddle'')  $\phi_0(x,y)=xy,$ where we would simply  get $2(y_2-y_1)(x_2-x_1)$ in place
of
\eqref{TV1}.

Since $z=(x,y)\in U_1\cup U_2,$ it will be particularly important to look at the
expression
\eqref{TV1} when $z=z_1\in U_1,$ and $z=z_2\in U_2.$  As above, if $U_1$ and
$U_2$ are sufficiently small, we can actually reduce to this case. We then see that  for
our perturbed saddle, still  the difference $y_2-y_1$  in the  $y$-coordinates plays an
important role as for the unperturbed saddle, but  in place  of  the  difference $x_2-x_1$
in the  $x$-coordinates now the quantities
\begin{align}\label{TV2}
	\tau_{z_1}(z_1,z_2):=x_2-x_1+y_2(y_2-y_1)\\
	\tau_{z_2}(z_1,z_2):=x_2-x_1+y_1(y_2-y_1)  \label{TV3}
\end{align}
become relevant.   Observe also that
 \begin{align}\label{symtrans}
 	\tau_{z}(z_1,z_2)=-\tau_{z}(z_2,z_1).
\end{align}
It is important to notice that the  constants $C(U_1,U_2)$ in the bilinear estimates
\eqref{bil1} will strongly depend on the sizes of the quantities appearing in \eqref{TV1}
-
\eqref{TV3}, as well as on the size of the derivatives of $\phi.$
\smallskip

 In the case of the saddle, since the ``transversality'' is here given by
 $2(y_2-y_1)(x_2-x_1),$ it is natural to perform a kind of Whitney decomposition with
 respect  to the diagonal of  $Q\times Q$ into direct products  $U_1\times U_2$  of  pairs
 of bi-dyadic rectangles  $U_1,U_2$  of the same dimension $\lambda_1\times \lambda_2,$
 which are separated in each coordinate by a distance proportional to the  sizes
 $\lambda_1$  respectively $\lambda_2,$   and then apply a  scaling transformation of the
 form $(x,y)=(\lambda_1 x',\lambda_2 y')$ -- this is exactly what had been done in
 \cite{lee05} and \cite{v05}.  Since the phase $xy$ is homogeneous under such kind of
 scalings,  in the new coordinates $(x',y'),$  one has then reduced the bilinear estimates
 to the case of normalized patches $U_1,U_2$ of size $1\times 1$ for which the
 transversalities are also of size $1,$ and from there on one could essentially apply the
 ``uniform'' bilinear estimates (similar to those known from the elliptic
 case) and re-scale them to go back to the original coordinates.
\smallskip

Coming back to our perturbed saddle, we shall again try to scale in order to make both
transversalities become of size $\sim\pm1.$  However,  as it will turn out, the ``right''
patches $U_1,U_2$  to be used will in general  no longer be bi-dyadic rectangles,  but in
some case a dyadic {\it square} $U_1$ and a  bi-dyadic {\it curved  box} $U_2.$
The ``right''  scalings,  which  will reduce matters to situations where all
transversalities  are of size $\sim\pm1,$  will be anisotropic, and this will  create a
new problem: while the saddle is invariant under  such type of scaling (i.e., for the
function $\phi_0(x,y):=xy,$ we have $\frac1{\lambda_1\lambda_2}\phi_0(\frac
x{\lambda_1},\frac y{\lambda_2})=\phi_0(x,y)),$ for our function $\phi,$ the  scaled
function   $\phi^s(x,y):=\frac1{\lambda_1\lambda_2}\phi(\frac x{\lambda_1},\frac
y{\lambda_2})$ is given by
$$ \phi^s(x,y)=xy+\frac {\lambda_2^2 y^3}{3\lambda_1}.
$$
Thus, if $\lambda_1\gg \lambda_2^2,$ then the second term can indeed be view as a small
perturbation of the leading term $xy,$ and we can proceed in a very similar way as for the
saddle. However, if $\lambda_1\lesssim\lambda_2^2,$ then the second, cubic  term, can
assume
a dominant role, and the treatment of this case will require further arguments.

In conclusion, the na\"\i f approach which would try to treat our surface $S$ as a
perturbation of the saddle, and which does
indeed work for   Stein-Tomas type   $L^q-L^2$  restriction estimates,  breaks down if we
want to derive restriction estimates of more general type by means of the bilinear method.

\medskip
In order to get some better idea on how to  suitably devise  the ``right'' pairs  of
patches $U_1, U_2$  for our Whitney type decomposition, note that the quantities
$\tau_{z_2}(z_1,z_2)$ and $\tau_{z_1}(z_1,z_2)$ can be of quite  different  size. For
instance, for $z_1^0=(0,0)$ and $z_2^0=(-1+\delta,1),$ we  have
$\tau_{z_1^0}(z_1^0,z_2^0)=\delta$ while $\tau_{z_2^0}(z_1^0,z_2^0)=-1.$ Therefore there
can be  a strong imbalance in the two  transversalities for the perturbed saddle.
This is quite different from the situation of the saddle, where the two quantities
\eqref{TV2}, \eqref{TV3} are the same, namely $x_2-x_1.$

Nevertheless,  observe that, due to the following  important  relation  between the  two
transversalities
\begin{align}\label{TV2+TV3}
\tau_{z_1}(z_1,z_2)-\tau_{z_2}(z_1,z_2)=(y_2-y_1)^2
\end{align}
(which is immediate   from  \eqref{TV2}, \eqref{TV3}),
we see that  at least one of the two transversalities  $\tau_{z_1}(z_1,z_2)$ or
$\tau_{z_2}(z_1,z_2)$ cannot be smaller than $|y_2-y_1|^2/4.$
This naturally leads to two cases, which we will discuss in detail in the next chapter: Either
$\tau_{z_1}(z_1,z_2)$ and $\tau_{z_2}(z_1,z_2)$ have similar size, and thus both are not much smaller than
$|y_2-y_1|^2$, or they are of quite different size, in which case one of the two transversalities has to be
comparable to $|y_2-y_1|^2$, while the other one can be much smaller.

\smallskip
In Section \ref{sect:admissible} we shall make a finer analysis of the transversality
conditions and introduce
the pairs of sets fit to them, \it admissible pairs\rm. In Section \ref{sect:bilin} we shall
state and prove the
sharp bilinear estimates for those pairs. In Section \ref{whitn} we will use the admissible
pairs to build a
Whitney type decomposition of $Q\times Q.$

\smallskip
Further difficulties arise in the passage from bilinear to linear Fourier extension, which
will  be discussed in Section \ref{bilinlin}.  In order to exploit all  of the underlying
almost orthogonality, we shall have to further decompose the curved boxes $U_2$ (whenever
they appear) into smaller squares and also make use of some disjointness properties of the
corresponding pieces of the surface $S.$  In contrast to what is done in the case of
elliptic surfaces, as well as for the saddle, it turns out that for the perturbed saddle
it
is not sufficient to exploit disjointness properties with respect to the first two
coordinates, but also with respect to the third one. A further novelty is that we have to
improve on a by now standard   almost orthogonality  relation  in $L^p$ between  the
pieces
arising in our Whitney type decomposition, which needs to be employed before applying our
bilinear estimates to each of these pieces. As it turns out, this ``classical''  estimate
is insufficient for our  curved boxes, and we improve on it by applying a  classical
square
function estimate  associated to  partitions into rectangular boxes,   due to Rubio de
Francia, which has had its roots in a square function estimate  obtained independently by
L. Carleson \cite{c67}, and A. Cordoba \cite{co81}.
Finally, we can reassemble the smaller squares and pass back to curved boxes by means of
Khintchine's inequality.

\medskip

\noindent{\bf Guide to the reader:} The real thrust of the  precise definition of admissible pairs given  in
Subsection \ref{preciseadmissible} will become relevant in an essential way only later in Section
\ref{whitn}, when we shall show that our
admissible pairs $(U_1,U_2)$ will allow to perform a Whitney-type decomposition of  $Q\times Q.$ To a smaller
degree, they are also relevant  in Subsection \ref{scaling transform},  which prepares for the  reduction of our
general bilinear estimates for admissible pairs in Theorem \ref{bilinear2} (for the curved box case)    to the
crucial prototypical case  introduced in Subsection \ref{proto} and studied in  Theorem \ref{bilinear}.

For a first reading, we therefore suggest to skip Subsections \ref{preciseadmissible} and  \ref{scaling
transform}, as well as the reduction of Theorem \ref{bilinear2}  to Theorem \ref{bilinear} in Section
\ref{sect:bilin}, and first read Subsections \ref{TVS} and \ref{proto}, and then in Subsection \ref{bilinarg}
the  proof of the bilinear estimates of Theorem \ref{bilinear},  which deals with the  prototypical case,
before
coming back to Subsection \ref{preciseadmissible}.

\medskip
\noindent\textsc{Convention:}
Unless  stated otherwise, $C > 0$ will stand for an absolute constant whose value may vary
from occurrence to occurrence. We will use the notation $A\sim_c B$  to express  that
$\frac{1}{c}A\leq B \leq c A$.  In some contexts where the size of $c$ is irrelevant we
shall drop the index $c$ and simply write $A\sim B.$ Similarly, $A\lesssim B$ will express
the fact that there is a constant $c$ (which does not depend  on the relevant quantities
in
the estimate) such that $A\le c B,$  and we write  $A\ll B,$ if the constant $c$ is
sufficiently small.
\medskip

\begin{center}
{\textsc{Acknowledgments}}
\end{center}
Part of this work was developed during the stay of the second and third author at the
Mathematical Sciences Research Institute at Berkeley during the Harmonic Analysis Program
of 2017. They wish to express their gratitude to the organizers and to the Institute  and
its staff for their hospitality and for providing a wonderful  working atmosphere.
\smallskip

We would also like to express our sincere gratitude to the referee for  many valuable  suggestions which have
greatly helped
to improve the presentation of the material in this article.


\setcounter{equation}{0}
\section{Transversality conditions and admissible pairs of sets}\label{sect:admissible}

\subsection{Admissible pairs of sets $U_1,$ $U_2$ on which transversalities are of a fixed
size: an informal discussion}\label{pairs of sets}

Recall again from the introduction  that the crucial
``transversality quantities''  arising in
Lee's estimate \eqref{trans} are given by $y_2-y_1$ and \eqref{TV2}, \eqref{TV3}, i.e.,
\begin{align*}
	\tau_{z_1}(z_1,z_2):=x_2-x_1+y_2(y_2-y_1), \\
	\tau_{z_2}(z_1,z_2):=x_2-x_1+y_1(y_2-y_1).
\end{align*}

We shall therefore try to devise neighborhoods $U_1$ and $U_2$ of two given points
 $z_1^0=(x_1^0,y_1^0)$ and $z_2^0=(x_2^0,y_2^0)$ on which these quantities are roughly
  constant for $z_i=(x_i,y_i)\in U_i,$ $i=1,2$,
 and which  are also  essentially chosen as large as possible. The corresponding pair $(U_1,U_2)$ of
 neighborhoods of $z^0_1$ respectively $z^0_2$ will be called an {\it admissible pair}.

The goal of this subsection is to present  some of the basic ideas, without being precise about  details, such
as
constants that will be hidden in the arguments, in order to  motivate the precise definition of admissible pairs
that will be given  in the next subsection (which might otherwise appear   a bit  strange).

\medskip

 In a first step, we choose a large constant $C_0\gg 1$, which will be made precise only later,
 and assume that
 $|y^0_2-y^0_1|\sim C_0\rho$ for some $\rho>0.$  It is then natural to allow $y_1$  to vary on  $U_1$ and $y_2$
 on $U_2$  by at most $\rho$ from $y^0_1$ and
 $y^0_2,$ respectively, i.e., we shall assume that
  \begin{align*}
|y_i-y^0_i|\lesssim \rho, \qquad \text{for}\quad  z_i\in U_i,\, i=1,2,
\end{align*}
so that indeed
\begin{equation}\label{rhosize}
|y_2-y_1|\sim C_0\rho \qquad\text{for}\quad  z_i\in U_i,\, i=1,2.
\end{equation}

 Recall next  the  identity \eqref{TV2+TV3}, which in particular implies that
  \begin{align}\label{TV2+TV3'}
|\tau_{z^0_1}(z^0_1,z^0_2)-\tau_{z^0_2}(z^0_1,z^0_2)|{\sim C_0^2}\rho^2.
\end{align}
We begin with

\medskip

\noindent {\bf Case 1: Assume that $ |\tau_{z_1^0}(z_1^0,z_2^0)|\le |\tau_{z_2^0}(z_1^0,z_2^0)|.$ }
Let us then write
\begin{equation}\label{defdelta}
 |\tau_{z_1^0}(z_1^0,z_2^0)|= \rho^2\de,
\end{equation}
where $\de\ge 0.$  Note, however, that  obviously $ \rho^2\de\lesssim1.$  From \eqref{TV2+TV3'} one then easily
deduces that there are two subcases:

\medskip

\noindent {\bf Subcase 1(a): (the ``straight box'' case),} where  $|\tau_{z_1^0}(z_1^0,z_2^0)| \sim
|\tau_{z_2^0}(z_1^0,z_2^0)|,$  or, equivalently, $\de\gtrsim1.$ In this case,  also
$|\tau_{z_2^0}(z_1^0,z_2^0)|\sim
\rho^2\de.$

\medskip

\noindent {\bf Subcase 1(b): (the ``curved box'' case),} where  $|\tau_{z_1^0}(z_1^0,z_2^0)| \ll
|\tau_{z_2^0}(z_1^0,z_2^0)|,$  or, equivalently, $\de\ll 1.$ In this case,  $|\tau_{z_2^0}(z_1^0,z_2^0)|\sim
\rho^2.$

\medskip
Given $\rho$ and $\delta,$ we shall then want to devise $U_1$ and $U_2$ so that the same kind of conditions hold
for all $z_1\in U_1$ and $z_2\in U_2,$
i.e.,
$$
 |\tau_{z_1}(z_1,z_2)|\sim \rho^2\de,\text{  and} \quad  |\tau_{z_2}(z_1,z_2)|\sim \rho^2(1\vee \de).
$$
Note that in view of  \eqref{TV2+TV3} and \eqref{rhosize} the second condition is redundant,
and so the  only additional condition that needs  to be satisfied  is that, for all
$z_1=(x_1,y_1)\in
U_1$ and $z_2=(x_2,y_2)\in U_2,$ we have
\begin{align*}
|\tau_{z_1}(z_1,z_2)|=|x_2-x_1+y_2(y_2-y_1)|\sim \rho^2\de.
\end{align*}
{As said before, }we want to choose $U_2$  as large  as possible w.r.  to $y_2,$ i.e., we only assume  that
$|y_2-y_2^0|\lesssim\rho.$ Let
\begin{equation}\label{a0}
a^0:=\tau_{z_1^0}(z_1^0,z_2^0)=x^0_2-x^0_1+y^0_2(y^0_2-y^0_1),
\end{equation} so  that $|a^0|\sim
\rho^2\de.$ Then we shall assume that {for $z_2\in U_2$} we have, say,
$|\tau_{z_1^0}(z_1^0,z_2)-\tau_{z_1^0}(z_1^0,z_2^0)|=|\tau_{z_1^0}(z_1^0,z_2)-a^0|\ll \rho^2\de.$  This means
that we shall choose $U_2$ to be
of the form
\begin{equation}\label{U2def}
U_2=\{(x_2,y_2): |y_2-y_2^0|\lesssim \rho,\ |x_2-x^0_1+y_2(y_2-y_1^0)- a^0|\ll \rho^2\de\}.
\end{equation}

As for $U_1,$ given our choice of $U_2,$ what we  still need is that
$|\tau_{z_1}(z_1,z_2)-\tau_{z_1^0}(z_1^0,z_2)|\ll \rho^2\de$ for all $z_1\in U_1$ and $z_2\in U_2,$ for then
also
$|\tau_{z_1}(z_1,z_2)-\tau_{z_1^0}(z_1^0,z_2^0)|\ll \rho^2\de$ for all such $z_1, z_2.$ We therefore  must
require that  $|x_1-x_1^0+y_2(y_1-y_1^0)|\ll \rho^2\de$ on $U_1\times U_2.$
But, since $y_2$ is allowed to vary within an interval of  size $\rho,$  we
 see that this requires a condition of the form $\rho|y_1-y_1^0|\ll\rho^2\de.$

Assuming this,
 we next  see that we are allowed to replace $y_2$ in the condition $|x_1-x_1^0+y_2(y_1-y_1^0)|\ll \rho^2\de$
 by
 $y_1^0.$ This  leads to the condition  $|x_1-x_1^0+y_1^0(y_1-y_1^0)|\ll \rho^2\de,$ which  depends on $z_1$
 only
 and thus gives our  second conditions on $U_1.$

 Our discussion suggests that we should finally choose $U_1$  of the form
\begin{equation}\label{U1def}
U_1=\{(x_1,y_1): |y_1-y_1^0|\lesssim \rho(1\wedge \de), |x_1-x_1^0+y_1^0(y_1-y_1^0)|\ll
\rho^2\de\}.
\end{equation}

\noindent{\bf Note:}  $U_1$ is essentially the affine image of  a rectangular box of dimension $\rho^2\de \times
\rho(1\wedge \de).$  However, when $\de\ll 1,$ then $U_2$ is a thin curved box, namely the segment of a
$\rho^2\de$-neighborhood of a parabola lying within the  horizontal strip where $|y_2-y_2^0|\lesssim \rho.$ On
the other hand, when $\de\gtrsim 1,$   then it is easily  seen that  $U_2$ is essentially a rectangular box  of
dimension $\rho^2\de\times \rho.$ This explains  why we  called Subcase 1(b)  where $\de\ll1$  the ``curved box
case'', and   Subcase 1(a)  where $\de\gtrsim 1$  the ``straight  box case.''

\medskip
\noindent {\bf Case 2: Assume that $|\tau_{z_1^0}(z_1^0,z_2^0)|\ge |\tau_{z_2^0}(z_1^0,z_2^0)|.$ }

This case  can easily be reduced to the previous one by symmetry. Indeed, in view of \eqref{symtrans}, we just
need to interchange the roles of $z_1$ and $z_2$ in the previous discussion, so that it is natural here to
define
an {\it admissible pair} $(\tilde U_1,\tilde U_2)$  {\it of type 2}  of neighborhoods $\tilde U_1$ of $z^0_1$
respectively $\tilde U_2$ of $z^0_2$ by setting

\begin{eqnarray}\label{tildeU}
  \begin{split}
\tilde U_1=\{(x_2,y_2): |y_1-y_1^0|\lesssim \rho,\ |x_1-x^0_2+y_1(y_1-y_2^0)- \tilde a^0|\ll \rho^2\de\},\\
\tilde U_2=\{(x_1,y_1): |y_2-y_2^0|\lesssim \rho(1\wedge \de), |x_2-x_2^0+y_2^0(y_2-y_2^0)|\ll
\rho^2\de\},
 \end{split}
\end{eqnarray}
where $\tilde a^0:=\tau_{z_2^0}(z_2^0,z_1^0).$

\bigskip

\medskip

\subsection{Precise definition of admissible pairs within $Q\times Q$}\label{preciseadmissible}
In view of our discussion in the previous subsection, we shall here  devise  more precisely
certain \lq\lq dyadic''    subsets   of $Q\times Q$ which will
assume the roles of the sets $U_1,$ respectively $U_2, $ in such a way  that on every pair of such sets each of
our transversality functions  is essentially of some fixed dyadic size, and which will moreover  lead to  a kind
of Whitney decomposition of $Q\times Q$ (as will be shown in  Section \ref{whitn})

To begin with, as before we fix a large dyadic constant $C_0\gg1.$

\smallskip
In  a first step,  we   perform a  classical {\bf dyadic decomposition in the
$y$-variable}
which is a variation of the one in \cite{TVV}:
For a given  dyadic number $0<\rho\lesssim 1,$  we denote for
$j\in\mathbb Z$ such that $|j|\rho\le 1$  by $I_{j,\rho}$ the dyadic interval
$I_{j,\rho}:=[j\rho,j\rho+\rho)$ of length $\rho,$ and by $V_{j,\rho}$ the corresponding
horizontal ``strip''   $V_{j,\rho}:=[-1,1]\times I_{j,\rho}$ within $Q.$
Given two dyadic intervals $J,\,J'$ of the same size, we say that they are
{\it related} if their parents are adjacent but they are not adjacent. We divide each
dyadic interval $J$ in a disjoint union of dyadic subintervals $\{I_J^k\}_{1\le k\le
C_0/8},$ of length  $8|J|/C_0.$  Then, we define $(I,I')$ to be an \it
admissible pair of dyadic intervals \rm if and only if there are $J$ and $J'$ related
dyadic intervals and $1\le k,\,j\le C_0/8$ such that $I=I_J^k$ and $I'=I_{J'}^j.$

We say that a pair of strips $(V_{j_1,\rho},V_{j_2,\rho})$ is {\it admissible } and  write
$V_{j_1,\rho}\backsim V_{j_2,\rho},$ if  $(I_{j_1,\rho},I_{j_2,\rho})$ is a pair of
admissible dyadic intervals. Notice that in this case,
  \begin{align}\label{admissibleV}
C_0/8< |j_2-j_1|< C_0/2.
\end{align}
One can  easily see that  this leads to the following disjoint decomposition of $Q\times
Q:$
\begin{align}\label{whitney1}
Q\times Q= \overset{\cdot}{\bigcup\limits_{\rho}}
\,\Big(\overset{\cdot}{\bigcup\limits_{V_{j_1,\rho}\backsim
V_{j_2,\rho}}}V_{j_1,\rho}\times V_{j_2,\rho}\Big),
\end{align}
where the first union is meant to be over all such dyadic $\rho$'s.

\medskip
In a second step, we perform a non-standard {\bf Whitney type decomposition of any given
admissible pair of strips}, to obtain subregions in which the transversalities are roughly
constant.

To simplify notation, we fix $\rho$ and an admissible pair  $(V_{j_1,\rho},V_{j_2,\rho}),$
and simply write $I_i:=I_{j_i,\rho},\, V_i:=V_{j_i,\rho}, \, i=1,2,$
so that $I_i$ is an interval of length $\rho$ with left endpoint  $j_i\rho,$   and
\begin{align}\label{Vi}
V_1=[-1,1]\times I_1, \qquad V_2=[-1,1]\times I_2,
\end{align}
are rectangles of dimension $2\times \rho,$ which are vertically  separated at scale
$C_0\rho.$
More precisely, for $z_1=(x_1,y_1)\in V_1$ and $z_2=(x_2,y_2)\in V_2$ we have
$|y_2-y_1|\in
|j_2\rho-j_1\rho|+[-\rho,\rho],$ i.e.,
\begin{align}\label{yseparation}
	C_0\rho/2\le |y_2-y_1|\le  C_0\rho.
\end{align}

Let $0<\delta\lesssim \rho^{-2}$ be a dyadic number (note that $\delta$ could be big,
depending on $\rho$), and  let $\I$ be  the set of points
which partition the interval $I$ into (dyadic) intervals of the same length  $\rho^2\de.$

\smallskip
Similarly, for $i=1,2,$ we choose a finite  equidistant partition $\I_i$  of  width
$\rho(1\wedge\delta)$  of  the interval $I_i$  by  points $y_i^0\in \I_i.$
Note: if $\de>1,$ then $\rho(1\wedge\delta)=\rho,$ and we can choose for $\I_i$ just
 the
 singleton  $\I_i=\{y_i^0\},$ {where $y_i^0$ is the left endpoint of $I_i.$}

 \medskip

\begin{definr}
For any parameters $x^0_1,t^0_2\in\I,$  $y^0_1\in\I_1$  defined in the previous lines and $y^0_2$ the left
endpoint of $I_2,$ we  define the
sets
\begin{eqnarray}\label{whitneybox1}
 \begin{split}
&U_1^{x^0_1,y_1^0,\delta} :=\{(x_1,y_1): 0\le y_1-y_1^0< \rho(1\wedge\delta),\,
0\le x_1-x^0_1+y_1^0(y_1-y_1^0)< \rho^2\delta \}, \\\\
&U_2^{t^0_2,y_1^0,y^0_2,\delta} :=\{(x_2,y_2): 0\le y_2-y^0_2<\rho, 0\le x_2-t^0_2+y_2(y_2-y_1^0)<
\rho^2\delta\},
\end{split}
\end{eqnarray}
and the points
\begin{equation}\label{pointsinU}
z^0_1:=(x^0_1,y^0_1), \qquad z^0_2=(x^0_2,y^0_2):=(t_2^0-y_2^0(y_2^0-y_1^0), y_2^0).
\end{equation}
\end{definr}

Observe that then
$$
z^0_1\in U_1^{x^0_1,y_1^0,\delta}\subset V_1\quad \text{  and   } \quad z^0_2\in
U_2^{t^0_2,y_1^0,y^0_2,\delta}\subset V_2.
$$
Indeed, $z^0_i$ is in some sense the ``lower left'' vertex of $U_i,$ and the horizontal projection of
$U_2^{t^0_2,y_1^0,y^0_2,\delta}$ equals $I_2.$
Moreover, if we define $a_0$ by \eqref{a0}, we have that $x_1^0+a^0=t^0_2,$ so that our definitions of the sets
$U_1^{x^0_1,y_1^0,\delta}$ and $U_2^{t^0_2,y_1^0,y^0_2,\delta}$ are very close to the ones for the  sets $U_1$
and
$U_2$ (cf. \eqref{U1def}, \eqref{U2def}) in the previous subsection.

 In particular, $U_1^{x^0_1,y_1^0,\delta}$ is  again essentially a paralellepiped of sidelengths $\sim
 \rho^2\de\times \rho(1\wedge \de),$ containing the point $(x^0_1,y_1^0),$  whose
longer side has slope $y_1^0$ with respect to the $y$-axis.
 Similarly, if $\de\ll 1,$
then $U_2^{t^0_2,y_1^0,y^0_2\delta}$ is a thin curved box of width $\sim\rho^2 \de$ and
length $\sim\rho,$  contained in a rectangle of dimension $\sim \rho^2\times \rho$ whose axes are
parallel to the coordinate axes (namely the part of a $\rho^2\delta$-neighborhood of a parabola containing the
point $(x^0_1,y^0_1)$ which lies  within the horizontal strip  $V_2$).   If $\de\gtrsim 1,$ then
$U_2^{t^0_2,y_1^0,y^0_2\delta}$ is essentially  a rectangular  box of dimension $\sim \rho^2\de\times \rho$
lying
in the same horizontal strip.

Note also that we have chosen to use  the parameter $t^0_2$ in place of using $x^0_2$ here, since with this
choice the identity
\begin{equation}\label{t2mean}
\tau_{z_1^0}(z_1^0,z_2^0)=t^0_2-x_1^0
\end{equation}
holds true, which will become quite useful in the sequel. We next have to relate the parameters $x_1^0,t^0_2,
y_1^0,y_2^0$ in order to give a precise definition of an admissible pair.

\smallskip
 Here, and in the sequel, we shall always assume that the points $z_1^0,z_2^0$ associated to these parameters
 are given by \eqref{pointsinU}.

\smallskip

\begin{definr}
Let us call a pair
 $(U_1^{x_1^0,y_1^0,\delta},U_2^{t_2^0,y_1^0,y_2^0,\delta})$ an {\it admissible
 pair of type 1
 (at scales $\de,\;\rho$ and contained in $V_1\times V_2$),} if the following two conditions hold
 true:
\begin{align}\label{admissible1}
\frac {C_0^2}4\rho^2\de\le
|\tau_{z_1^0}(z_1^0,z_2^0)|&=|t_2^0-x_1^0|<4 \,C_0^2\rho^2\de,\\
\frac{C_0^2}{512}\rho^2(1\vee
\de)\le|\tau_{z_2^0}(z_1^0,z_2^0)|&< 5\,
C_0^2\rho^2(1\vee \de).\label{admissible2}
\end{align}
By $\cP^{\de}$ we shall denote the set of all admissible pairs of
type 1 at scale $\de$ (and  $\rho$, contained in $V_1\times V_2,$), and by $\cP$
the  corresponding union over all dyadic
scales $\de.$ 
\end{definr}

Observe that, by \eqref{TV2+TV3}, we have
$\tau_{z_2^0}(z_1^0,z_2^0)=\tau_{z_1^0}(z_1^0,z_2^0)-(y_2^0-y_1^0)^2.$
In view of   \eqref{admissible1} and \eqref{yseparation} this shows that condition
\eqref{admissible2} is automatically satisfied, unless $\de\sim 1.$

We remark that it would indeed be more appropriate to denote the sets $\cP^{\de}$ by $\cP^{\de}_{V_1\times V_2},$  but we
want to simplify the notation. In all instances in the rest of the paper $\cP^{\de}$ will be  associated  to a
fixed admissible pair of strips $(V_1,V_2),$ so  that our imprecision will not cause any ambiguity.

\begin{lemnr}\label{sizeofdeltas}
If  $(U_1^{x_1^0,y_1^0,\delta},U_2^{t_2^0,y_1^0,y_2^0,\delta})$ is an admissible
pair of type 1, then for
all  $(z_1,z_2)\in (U_1^{x_1^0,y_1^0,\delta},U_2^{t_2^0,y_1^0,y_2^0,\delta})$ ,
$$
|\tau_{z_1}(z_1,z_2)|\sim_{8} C_0^2 \rho^2\de\mbox{   and   }
|\tau_{z_2}(z_1,z_2)|\sim_{1000} C_0^2 \rho^2(1\vee\de).
$$
\end{lemnr}

\begin{proof}
Note that
\begin{eqnarray}\label{tauid}
&&\tau_{z_1}(z_1,z_2)=x_2-x_1+y_2(y_2-y_1)=t_2^0-x_1^0\\
&&\hskip0.5cm+ \big[x_2-t_2^0+y_2(y_2-y_1^0)\big]-
\big[x_1-x_1^0+y_1^0(y_1-y_1^0)\big]+(y_1^0-y_2)(y_1-y_1^0),
\nonumber
\end{eqnarray}
where, by \eqref{yseparation} and our definition of $U_1^{x_1^0,y_1^0,\delta}$ and
$U_2^{t_2^0,y_1^0,y_2^0,\delta},$ we have
$|x_2-t_2^0+y_2(y_2-y_1^0)|<\rho^2\de,\  \big|(x_1-x_1^0)+y_1^0(y_1-y_1^0)\big|<\rho^2\de$
and
   $|(y_1^0-y_2)(y_1-y_1^0)|<C_0\rho \cdot\rho(1\wedge\delta)\le C_0\rho^2\de.$
This shows that
\begin{align}\label{tausize1}
|\tau_{z_1}(z_1,z_2)-\tau_{z_1^0}(z_1^0,z_2^0)|\le 2C_0
\rho^2\de,
 \end{align}
 and in particular in combination with \eqref{admissible1} that
 $|\tau_{z_1}(z_1,z_2)|\sim_{8} C_0^2 \rho^2\de,$
if we choose $C_0$ sufficiently large.

 Similarly,  because of \eqref{TV2+TV3'}, we have
$$
\tau_{z_2}(z_1,z_2)-\tau_{z_2^0}(z_1^0,z_2^0)=\tau_{z_1}(z_1,z_2)-\tau_{z_1^0}(z_1^0,z_2^0)
-(y_2-y_1)^2+(y_2^0-y_1^0)^2,
$$
where
\begin{align*}
|-(y_2-y_1)^2+(y_2^0-y_1^0)^2|&=|(y_2^0-y_2)+(y_1-y_1^0)|\,|(y_2^0-y_1^0)+(y_2-y_1)|\\
&\le 2\rho \cdot 2 C_0\rho=4 C_0\rho^2.
\end{align*}

In combination with \eqref{tausize1} this implies
\begin{align}\label{tausize2}
|\tau_{z_2}(z_1,z_2)-\tau_{z_2^0}(z_1^0,z_2^0)|\le
6C_0 \rho^2(1\vee\de).
 \end{align}
Invoking also  \eqref{admissible2} this implies $|\tau_{z_2}(z_1,z_2)|\sim_{1000} C_0^2
\rho^2(1\vee\de).$
\end{proof}

\medskip

Finally, as in Case 2 of  the previous subsection, we also need to consider the symmetric case and define
admissible pairs where
$|\tau_{z_1}(z_1,z_2)|\gtrsim|\tau_{z_2}(z_1,z_2)|.$ By
interchanging the roles of $z_1$ and $z_2$ (compare also with \eqref{symtrans})  we define
accordingly  for any $t^0_1, x^0_2\in\I,$  $y^0_1$ the left endpoint of $I_1$ and $y^0_2\in\I_2$  the sets
\begin{align*}
	{\tilde U}_1^{t^0_1,y_1^0,y_2^0,\delta} &:=\{(x_1,y_1): 0\le y_1-y^0_1<\rho, 0\le x_1-t^0_1+y_1(y_1-y_2^0)<
\rho^2\delta\},\\
	{\tilde U_2}^{x^0_2,y_2^0,\delta} &:=\{(x_2,y_2): 0\le y_2-y_2^0< \rho(1\wedge\delta),\,
0\le x_2-x^0_2+y_2^0(y_2-y_2^0)< \rho^2\delta \}.
		\end{align*}
The corresponding points $z^0_1\in {\tilde U}_1^{t^0_1,y_1^0,y_2^0,\delta}$ and $z^0_2\in {\tilde
U_2}^{x^0_2,y_2^0,\delta}$ are here defined by
$$
 z^0_1=(x^0_1,y^0_1):=(t_1^0-y_1^0(y_1^0-y_2^0), y_1^0), \qquad z^0_2:=(x^0_2,y^0_2).
$$

In analogy to our previous definition,  if the  conditions
$$
\frac {C_0^2}4\rho^2\de\le
|\tau_{z_2^0}(z_1^0,z_2^0)|=|x_2^0-t_1^0|<4 \,C_0^2\rho^2\de$$
 (in place of \eqref{admissible1}) and
$$
C_0^2\rho^2(1\vee \de)/512\le|\tau_{z_1^0}(z_1^0,z_2^0)|< 5
C_0^2\rho^2(1\vee \de)
$$ (in place of \eqref{admissible2}) are satisfied, we shall call the
pair $({\tilde U}_1^{t^0_1,y_1^0,y_2^0,\delta} , {\tilde U_2}^{x^0_2,y_2^0,\delta})$ an
{\it
admissible pair of type 2 (at scales  $\de,\;\rho$ and  contained in $V_1\times V_2$)\rm}. Note that this means
that $({\tilde U_2}^{x^0_2,y_2^0,\delta}, {\tilde U}_1^{t^0_1,y_1^0,y_2^0,\delta})$ is an
admissible pair of type 1  in $V_2\times V_1$ at the same scales.
\medskip
By $\tilde \cP^{\de},$ we shall denote the set of all admissible pairs of
 type 2  at scale $\de$  (and $\rho$, contained in $V_1\times V_2,$), and by $\tilde \cP$
the  corresponding unions over all dyadic
scales $\de.$ 
\medskip

In analogy with Lemma \ref{sizeofdeltas}, we have 
\begin{lemnr}\label{sizeofdeltas2}
If $(\tilde U_1,\tilde U_2)=({\tilde U}_1^{t^0_1,y_1^0,y_2^0,\delta},{\tilde U_2}^{x^0_2,y_2^0,\delta})\in \tilde \cP^\de$  is an admissible pair of type 2, then for all $(z_1,z_2)\in(\tilde U_1,\tilde U_2)$ we have
$$
|\tau_{z_1}(z_1,z_2)|\sim_{1000} C_0^2 \rho^2(1\vee\de)\mbox{   and   }
|\tau_{z_2}(z_1,z_2)|\sim_8 C_0^2 \rho^2\de.
$$
\end{lemnr}

\subsection{The exact transversality conditions}\label{TVS}

In the curved box case, i.e., when $\de \ll 1,$   it  turns out that  one cannot directly reduce the bilinear
Fourier extension estimates over admissible pairs  to Lee's  Theorem 1.1 in \cite{lee05}, since that would not
give us the optimal  dependence on $\delta$.   We shall  therefore
have to be more precise about the required transversality conditions. So, let us   recall in more detail the
exact  transversality conditions mentioned in the Introduction  that we need for the bilinear  argument.  As
references to this (by now
standard) argument we refer for instance to \cite{lee05}, and  \cite{v05}.
\medskip

In  this bilinear argument, we assume that  we are given  two  patches of subsurfaces
$S_i=\{(z_i,\phi(z_i)): z_i\in U_i\}, i=1,2,$ of $S.$ For fixed  points $z_1'\in U_1$  and
$z_2'\in U_2,$ we consider the translated surfaces
$$
\tilde S_1:=S_1+(z_2',\phi(z_2'))\mbox{ and } \tilde S_2:=S_2+(z_1',\phi(z_1')).
$$
We will assume in this subsection that there is a constant $C>0$ such that
$|\nabla\phi(z)|\le C$ for all $z\in U_1\cup U_2.$ The implicit constants in the argument
will depend on this $C.$ We will not make any assumption about other derivatives of $\phi$
and will keep track of the dependence on them of the transversalities. That will be
important in the rest of the section.

\medskip

If the normals to these two  surfaces $S_i$ at the  points $(z_i',\phi(z_i')),$ $i=1,2,$
 are not parallel, we can locally define the {\it intersection curve}
$$
\Pi_{z_1',z_2'}:=\tilde S_1\cap\tilde
S_2=[S_1+(z_2',\phi(z_2'))]\cap[S_2+(z_1',\phi(z_1'))].
$$
Note that
$$
\tilde S_1=\{(z,\phi(z-z_2')+\phi(z_2')): z\in U_1+z_2'\} \mbox{ and }\tilde
S_2=\{(z,\phi(z-z_1')+\phi(z_1')): z\in U_2+z_1'\}.
$$
Set $\psi(z):= \phi(z-z_1')+\phi(z_1')-\phi(z-z_2')-\phi(z_2').$ Then, the orthogonal
projection of the curve $\Pi_{z_1',z_2'}$ on the $z$ - plane  is the  curve  given by
$\{z:\;\psi(z)=0\}.$
We introduce a parametrization by arc length $\gamma(t),\, t\in  J,$  of this curve, where
$t$ is from an  open interval $ J.$
Notice that $\gamma(t)$ depends on the choices of $z_1'$ and $z_2'.$
By $N(x,y)$ we denote the following normal to our surface $S$  at $(x,y,\phi(x,y))\in S:$
$N(x,y):=
\left(
\begin{array}{c}
  \trans\nabla\phi(x,y)   \\
-1
\end{array}
\right).
$
Note that these normal vectors are of size $|N(x,y)|\sim 1.$
Then  the vector $N_2(z):=N(z-z_1')$ is normal to the translated surface $\tilde S_2$ at
the point $(z,\phi(z_1-z_1')+\phi(z_1')),$  and we  consider the ``cone of normals of type
2 along the intersection curve''  $\Gamma_2:=\{sN_2(\gamma(t)):  s\in \mathbb R,\,t\in
J\}.$ In an analogous way, we define the ``cone $\Gamma_1$ of normals of type 1 along the
intersection curve''.

In the bilinear argument (see, for instance, \cite{v05}, final remark on page 110),  the
condition which  is needed is  that   the normal vectors to $S$ at all points of $S_1$ are
transversal to  the cone $\Gamma_2,$ more precisely that
\begin{align}\label{transv}
\left|\det\left(\frac{N(z_1)}{|N(z_1)|}\, \frac{N_2(\gamma(t))}{|N_2(\gamma(t))|}
\,\frac{\frac{d}{dt}N_2(\gamma(t))}{|\frac{d}{dt}N_2(\gamma(t))|}
\right)\right|\ge C>0,
\end{align}
for all $z_1\in U_1$ and all $t\in J,$ and that the symmetric condition holds true  when
the roles of $S_1$ and $S_2$ are interchanged, i.e.,  for $S_2$ and  $\Gamma_1.$ The above
determinant is equal to
\begin{align}\label{transv1}
	TV_2(z_2,z_1):= \frac{\det \left(\begin{array}{ccc}
      \trans\nabla\phi( z_1) & \trans\nabla\phi( z_2)& H\phi( z_2)\cdot\trans\omega  \\
    -1 & -1 & 0
\end{array}\right)}
{\sqrt{1+|\nabla\phi( z_1)|^2}\sqrt{1+|\nabla\phi(z_2)|^2}\,
|H\phi(z_2)\cdot\trans\omega|},
\end{align}
where $\omega=\gamma'(t),$ for $t$ such that $\gamma(t)=z_2+z_1'.$ A similar expression is
obtained for  $S_2$ and $\Gamma_1,$  namely
\begin{align*}
	TV_1(z_1,z_2):= \frac{\det \left(\begin{array}{ccc}
      \trans\nabla\phi( z_2) & \trans\nabla\phi( z_1)& H\phi( z_1)\cdot\trans\omega  \\
    -1 & -1 & 0
\end{array}\right)}
{\sqrt{1+|\nabla\phi( z_2)|^2}\sqrt{1+|\nabla\phi(z_1)|^2}\,
|H\phi(z_1)\cdot\trans\omega|},
\end{align*}
where here we assume that  $z_2\in U_2$ and $z_1\in U_1$ is such that
$z_1+z_2'=\gamma(s)$
for some $s\in J,$  and then
 $\omega:=\gamma'(s).$

\medskip

Condition \eqref{transv} can be written as
\begin{align}\label{transv5}
|TV_1(z_1,z_2)|,\; |TV_2(z_2,z_1)|\ge C>0.
\end{align}

\medskip

 Notice that, formally, $TV_2(z_2,z_1)=TV_1(z_2,z_1)$ (though, $z_1'$ and
 $z_2'$ should also be interchanged). Moreover, one  easily shows that
\begin{align}\label{transv3}
	TV_2(z_2,z_1)= \frac{(\nabla\phi( z_1)-\nabla\phi( z_2))\cdot J \cdot H\phi(
z_2)\cdot\trans\omega}
{\sqrt{1+|\nabla\phi( z_1)|^2}\sqrt{1+|\nabla\phi(z_2)|^2}\,
|H\phi(z_2)\cdot\trans\omega|},
\end{align}
where $J$ denotes the symplectic matrix
$J:=\left(
\begin{array}{cc}
   0 & 1    \\
 -1 &  0   \\
\end{array}
\right).$

\medskip

At this point, it is interesting to explain how the quantities $TV_1$ and $TV_2$ relate to
$\Gamma^\phi_{z}$ defined by \eqref{transs}. Going back to \eqref{transv1}, note that
\begin{align*}
\det\left(\begin{array}{ccc}
      \trans\nabla\phi( z_1) & \trans\nabla\phi( z_2)& H\phi( z_2)\cdot\trans\omega  \\
    -1 & -1 & 0
\end{array}\right)\qquad\qquad\qquad\qquad\qquad\qquad\qquad\qquad\qquad\qquad\\
=\det\left(\begin{array}{cc}
      H\phi( z_2) & 0  \\
    0 & 1
\end{array}\right)\det\left(\begin{array}{ccc}
      (H\phi( z_2))^{-1}\trans\nabla\phi( z_1) & (H\phi( z_2))^{-1}\trans\nabla\phi( z_2)&
      \trans\omega  \\
    -1 & -1 & 0
\end{array}\right) \\
=\det H\phi( z_2) \det\left(\begin{array}{ccc}
      (H\phi( z_2))^{-1}(\trans\nabla\phi( z_1)-\trans\nabla\phi( z_2)) & (H\phi(
      z_2))^{-1}\trans\nabla\phi( z_2)& \trans\omega  \\
   0 & -1 & 0
\end{array}\right).
\end{align*}
Set $z_1'' \in U_1,$ such that $\gamma(t)=z_2+z_1'=z_1''+z_2'.$ Then, $\omega=\gamma'(t)$
is unitary and orthogonal to $\nabla\phi( z_1'')-\nabla\phi( z_2).$
Hence,
\begin{align}\label{transv4}
|TV_2(z_2,z_1)|= \frac{\big|\det H\phi( z_2)\left\langle
(H\phi)^{-1}(z)(\nabla\phi(z_2)-\nabla\phi(z_1)),
\frac{\nabla\phi(z_2)-\nabla\phi(z_1'')}{|\nabla\phi(z_2)-\nabla\phi(z_1'')|}
\right\rangle\big|}{\sqrt{1+|\nabla\phi(
z_1)|^2}\sqrt{1+|\nabla\phi(z_2)|^2}\,|H\phi(z_2)\cdot\trans\omega|}\\
=\frac{\big|\det H\phi( z_2)\,\Gamma^\phi_{z_2}(z_1,z_2,z_1'',z_2)\big|}
{\sqrt{1+|\nabla\phi( z_1)|^2}\sqrt{1+|\nabla\phi(z_2)|^2}\,
|H\phi(z_2)\cdot\trans\omega||\nabla\phi(z_2)-\nabla\phi(z_1'')|}.
\end{align}

\medskip

\subsection{A prototypical admissible pair in the curved box case and  the crucial scaling
transformation}\label{proto}

In this section we shall present a  \lq\lq prototypical"
case where  $U_1$ and $U_2$ will form an admissible pair of type 1 centered at  $z_1^0=0\in U_1$ and $z_2^0\in
U_2,$  with  $\rho\sim1,$ i.e., $|y_1^0-y_2^0|\sim 1,$ and $|\tau_{z_2^0}(z_1^0,z_2^0)|\sim 1$  but
$|\tau_{z_1^0}(z_1^0,z_2^0)|\sim \delta\ll 1.$  This means that we shall be in the curved box case. Recall from
the Introduction the identity
given by \eqref{gammaz} and \eqref{TV1}. As we  will show in Subsection \ref{scaling transform} in
detail, we can
always reduce to this particular  situation
when the two transversalities  $\tau_{z_2^0}(z_1^0,z_2^0)$ and $\tau_{z_1^0}(z_1^0,z_2^0)$
are of quite different sizes.

Fix  a small number $0<c_0\ll1$ ($c_0=10^{-10}$ will, for instance, work). Assume that
$0<\delta\le1/10,$ and put
\begin{align}
	U_1:=&[0,c_0^2\delta)\times [0,c_0\delta) \label{U1}\\
	U_2:=&\{(x_2,y_2):0\le y_2-b< c_0,0\le x_2+y_2^2-a<c_0^2\delta\},\label{U2}
\end{align}
where $|b|\sim_21$ and $|a|\sim_4\delta$.

\

\noindent\bf Remark. \rm Note that, if we set $C_0=1/{c_0},$  $\rho=c_0,$ then any admissible pair
$(U_1,U_2)=(U_1^{0,0,\delta},U_2^{a,0,b,\delta})$ would satisfy \eqref{U1} and \eqref{U2} with the above conditions on $a$ and $b.$
\medskip

Observe  that for $z_1=(x_1,y_1)\in U_1$ and $z_2=(x_2,y_2)\in U_2,$  we have
$|y_1-y_2|\sim1.$
Moreover, $(0,0)\in U_1,$ and  $\tau_{(0,0)}((0,0),(x_2,y_2))\sim\delta$ for all
$(x_2,y_2)\in U_2,$ which easily implies that, more generally,
$\tau_{(x_1,y_1)}((x_1,y_1),(x_2,y_2))\sim\delta$ for all  $(x_1,y_1)\in U_1$ and
$(x_2,y_2)\in U_2.$ Moreover, $\tau_{(x_2,y_2)}((x_1,y_1),(x_2,y_2))\sim1$ for all
$(x_1,y_1)\in U_1$ and $(x_2,y_2)\in U_2.$
One easily  computes  also that
$TV_1(z_1,z_2)\sim \delta$  and $TV_2(z_2,z_1)\sim 1,$  respectively. Thus there is an
unlucky  discrepancy between those two transversalities, since  $\delta\ll 1,$  and a
straight-forward application of  the bilinear method  would lead to a worse dependency on
$\delta$  of the constant in the bilinear estimate for these sets than the estimate   \eqref{bilinest} in
Theorem \ref{bilinear}, which  we  shall need.\smallskip

In the following lines, we shall therefore apply a suitable scaling which will  turn both
transversalities to  become of size $\sim 1.$ The price, however, that we shall have to
pay
is that, after scaling, the curvature of one of the two patches of surface will become
large (compare \eqref{hessphis}), so that we still cannot apply standard bilinear
estimates, but shall have to go into more detail into the proof of those estimates. Those
details will be given in Subsection \ref{bilinarg}.

\medskip

To overcome the afore-mentioned problem, we introduce the scaling
$$\phi^s(\bar z):=\frac{1}{\mathfrak a}\phi(A\bar z),$$
 where $A$ is a regular matrix, $\mathfrak a$ real.
 Concretely, we choose $(x,y)=z=A\bar  z:=(\delta\bar x,\bar y),$ i.e.,
\begin{align*}
	\bar x&=\delta^{-1}x,\\
	\bar y&=y,
\end{align*}
and $\mathfrak a:=\det A=\delta$. Denote by $U_i^s:=A^{-1}U_i$ the re-scaled domains
$U_i,$
and  by $S_i^s$ the scaled surface patches
$$
S_i^s:=\{(\bar x,\bar y,\phi^s(\bar x,\bar y)): (\bar x,\bar y)\in U_i^s\}.
$$
Explicitly, we then have
$$
	\phi^s(\bar x,\bar y)=\bar x\bar y+\frac{1}{3\delta}\bar y^3,
$$
hence
\begin{align}\label{nablaphis}
	\nabla\phi^s(\bar x,\bar y)=(\bar y,\bar x+\frac{1}{\delta}\bar y^2)
	=(y,\frac{1}{\delta}(x+y^2))
\end{align}
and
\begin{align}\label{hessphis}
	H\phi^s(\bar x,\bar y)=\left(\begin{array}{cc}
    0  & 1  \\
    1 & 2\delta^{-1}\bar y
\end{array}\right)
=\left(\begin{array}{cc}
    0  & 1  \\
    1 & 2\delta^{-1} y
\end{array}\right).
\end{align}
In particular, we see that
\begin{align}\label{bddgrad}
	|\nabla\phi^s(\bar z)|\lesssim 1
\end{align}
for all $z=A\bar  z\in U_1\cup U_2$.

 Let us also put $\bar a:=a/\delta,\, \bar b:=b, $ so that  $|\bar b|\sim_{2}1$ and $|\bar
 a|\sim_4 1$ and choose
$z_i=A\bar  z_i\in U_i,$ $i=1,2.$

 Then,  by \eqref{U1}, \eqref{U2},   we have  that  $y_2-y_1=\bar b+\Landau(c_0)$ and
 $x_1+y_1^2\leq 2 c_0\delta$, while $(x_2+y_2^2)/\delta=\bar a+\Landau(c_0)$. We conclude
 by \eqref{nablaphis} that
\begin{align*}
	\nabla\phi^s(\bar z_2)-\nabla\phi^s(\bar z_1) =(\bar b, \bar a)+\Landau(c_0).
\end{align*}
This shows that  if  the vector $\omega=(\omega_1,\omega_2)$ is tangential to the intersection curve of
$ S_1^s$ and $S_2^s,$  then \begin{align}\label{tangent}
	\omega=(-\bar a,\bar b)+\Landau(c_0).
\end{align}
In  combination with \eqref{hessphis} this implies that
\begin{align*}
	H\phi^s(\bar z_i)\cdot\trans\omega = \trans (\bar b, -\bar a+2\bar
b\delta^{-1}y_i)+\Landau\big((1+|\delta^{-1}y_i|)c_0\big)
\end{align*}
and
\begin{align}\label{sep}
	|H\phi^s(\bar z_i)\cdot\trans\omega| \sim 1+|\delta^{-1}y_i|.
\end{align}

By \eqref{transv3}, the transversalities for the scaled patches of surface  $S_i^s,$ and
the scaled function, $\phi^s
\,i=1,2,$  are thus given by
\begin{align*}
	\Big|TV^s_i(\bar z_1,\bar z_2)\Big|=&\Big|(-1)^{i +1}\frac{(\nabla\phi^s(\bar
z_1)-\nabla\phi^s( \bar z_2))\cdot J \cdot H\phi^s( \bar z_i)\cdot\trans\omega}
{\sqrt{1+|\nabla\phi^s( \bar z_1)|^2}\sqrt{1+|\nabla\phi^s(\bar z_2)|^2}\, |H\phi^s(\bar
z_i)\cdot\trans\omega|}\Big| \\
	 \sim& \Big|\frac{2\bar b(-\bar a+\bar b \delta^{-1} y_i) +
\Landau\big((1+|\delta^{-1}y_i|)c_0)\big)}{ 1+|\delta^{-1}y_i|}\Big|.
\end{align*}
Now, if $i=1,$ then $|y_1|\leq c_0\delta$, i.e., $|\delta^{-1}y_i|\le c_0\ll1,$ so that
clearly $\Big|TV^s_1(\bar z_1,\bar z_2)\Big|\sim 1.$

And, if $i=2,$ then $|y_i-\bar b|\le c_0\ll1,$ which easily implies that $|2\bar b(-\bar
a+\bar b \delta^{-1} y_i) + \Landau\big((1+|\delta^{-1}y_i|)c_0\big)|\sim \delta^{-1},$
and
also $1+|\delta^{-1}y_i|\sim \delta^{-1},$ so that  $\Big|TV^s_2(\bar z_1,\bar
z_2)\Big|\sim 1.$

We have thus shown

\begin{lemnr}\label{transvscaled}
The transversalities for the scaled patches of surface  $S_i^s, \,i=1,2,$ satisfy
$$
\Big|TV^s_i(\bar z_1,\bar z_2)\Big|\sim 1, \quad i=1,2.
$$
\end{lemnr}

\medskip
\subsection{Reduction to the prototypical case}\label{scaling transform}
Let $(U_1,U_2)\in\cP^\de$ be an admissible pair of type 1, where
$U_1=U_1^{x^0_1,y_1^0,\delta}$ and $U_2= U_2^{t^0_2,y_1^0,y^0_2\delta}.$
In this section, we shall see that  the bilinear estimates associated to the  sets
$U_1,U_2$ can easily be
reduced by means of a suitable
affine-linear transformation to either the  classical bilinear estimate in \cite{lee05},
when $\delta\ge 1/10,$
or
to the estimate for the  special  ``prototype'' situation  given in  Subsection
\ref{proto}, when $\delta\le 1/10.$

\medskip

Recall from \eqref{pointsinU} that we then had put
$$
z^0_1:=(x^0_1,y^0_1)\in U_1, \qquad z^0_2=(x^0_2,y^0_2):=(t_2^0-y_2^0(y_2^0-y_1^0), y_2^0)\in U_2.
$$
We translate the point $z_1^0$ to  the origin. Under the corresponding
translation, the phase function $\phi$ changes as follows: 
\begin{align*}
	\phi(z_1^0+\tilde z)=&(x_1^0+\tilde x)(y_1^0+\tilde y)+\frac{1}{3}(y_1^0+\tilde y)^3 \\
	=& \tilde x\tilde y+\frac{1}{3}{\tilde y}^3+y_1^0 {\tilde y}^2+ \text{affine linear terms} \\
	=& (\tilde x+y_1^0 \tilde y)\tilde y+\frac{1}{3}{\tilde y}^3+ \text{affine linear terms}.
\end{align*}
It is therefore  convenient to  introduce new coordinates  $z''=(x'',y''),$ by putting
\begin{align}\label{changeofvariables1}
	x'':=& \tilde x+y_1^0 \tilde y = x-x_1^0 + y_1^0 (y-y_1^0), \nonumber\\
	y'':=&\tilde  y = y-y_1^0.
\end{align}

Then, in these new coordinates, the phase function is given by
\begin{align}\label{changeofvariables2}
	\tilde \phi(z'')=\phi(z'')+ \text{affine linear terms}.
\end{align}
Noticing  that affine linear terms in the phase function play no role in our Fourier
extension estimates, we may thus assume that $\tilde \phi=\phi.$
In view of the size of the sets $U_i,$ we perform   a further
scaling transformation by writing  $x''= \rho^2 (1\vee\delta)x',\,  y''=\rho y'.$
Then,
clearly
$\phi(x'',y'')=\rho^3 (1\vee\delta)\phi_\de(x',y'),$ if we put
$$
\phi_\delta(z):=xy+\frac{y^3}{3(1\vee\delta)}.
$$

Thus, altogether we define the  change of coordinates $z'=T(z)$ by
\begin{align*}
	x':=&(1\vee\delta)^{-1}\rho^{-2}(x-x_1^0+y_1^0(y-y_1^0)),\\
	y':=&\rho^{-1}(y-y_1^0).
\end{align*}

Notice that from the following lemma, in the case $\delta\le 1/10,$ it is easy to pass  to the prototypical case by another harmless scaling $(x',y')=(C_0^{2}x''',C_0y''').$ We  shall skip the details.

\begin{lemnr}\label{U'}
We have
\begin{align}\label{phi-phide}
	\phi(z)=\rho^3(1\vee\de)\phi_\de(Tz)+L(z),
\end{align}
where $L$ is an affine-linear map. Moreover, in these new coordinates, $U_1,U_2$
correspond
to the sets
\begin{align}\label{scaledsets1}
	U_1':=T(U_1)=&\{(x',y'):0\le y'< 1\wedge\delta,\, 0\le x'<  1\wedge\delta\}
=[0,1\wedge\delta)^2,\\
	U_2':=T(U_2)=&\{(x',y'):0\le y'-b< 1,x'+\frac{y'^2}{1\vee\delta}-a<
1\wedge\delta\}, \label{scaledsets2}
\end{align}
where $|b|:=|\rho^{-1}(y_2^0-y_1^0)|\sim_2 C_0$  and
$|a|:=|\rho^{-2}(1\vee\delta)^{-1}(t_2^0-x_1^0)|\sim_{4} C_0^2\frac{\delta}{1\vee\delta} =
C_0^2(1\wedge\delta).$

Moreover, for Lee's transversality expression $\Gamma^{\phi_\de}$  in \eqref{transs} for
$\phi_\de$ , we have that
\begin{align}\label{transscaled}
|\Gamma^{\phi_\de}_{\tilde z'_1}(z'_1,z'_2)|\sim  C_0^3(1\wedge \de)\quad \text{for all }
\tilde
z'_1\in U'_1, \quad |\Gamma^{\phi_\de}_{\tilde z'_2}(z'_1,z'_2)|\sim  C_0^3 \quad
\text{for
all }
\tilde z'_2\in U'_2,
\end{align}
for every $z'_1\in U'_1$ and every $z'_2\in U'_2.$ Also, for $\delta\ge 1/10,$ the derivatives
of
$\phi_\delta$ can be uniformly
(independently of $\delta$) bounded from above.
\end{lemnr}

\begin{proof}
The first identity \eqref{phi-phide} is clear from our previous discussion.

The identities  \eqref{scaledsets1}, \eqref{scaledsets2} and the  formulas for $a$ and $b$
follow by straight-forward computation, and the statements about the sizes of $a$ and $b$
follow from \eqref{yseparation} and \eqref{admissible1}.

Recall  that $\Gamma^\phi_{z_i}(z_1,z_2):=\left\langle
(H\phi)^{-1}(z_i)(\nabla\phi(z_2)-\nabla\phi(z_1)),\nabla\phi(z_2)-\nabla\phi(z_1)\right\rangle,$
and denote  by $\Gamma^{\phi_\de}_{z'_1}(z_1,z'_2)$ the corresponding quantity associated
to $\phi_\de.$ These are obviously related by
$$
\Gamma^{\phi_\de}_{z'_i}(z'_1,z'_2)=\frac
1{\rho^3(1\vee\de)}\Gamma^{\phi}_{T^{-1}z'_i}(T^{-1}z'_1,T^{-1}z'_2).
$$
Recall also from \eqref{TV1} that
$\Gamma^{\phi}_{z_1}(z_1,z_2)=2(y_2-y_1)\tau_{z_1}(z_1,z_2)$ and
$\Gamma^{\phi}_{z_2}(z_1,z_2)=2(y_2-y_1)\tau_{z_2}(z_1,z_2).$ Thus, by
Lemma \ref{covering} and \eqref{yseparation},  we have that for $z'_1\in U'_1, z'_2\in
U'_2,$
$$
|\Gamma^{\phi}_{T^{-1}z'_1}(T^{-1}z'_1,T^{-1}z'_2)|\sim \rho \, C_0^3\rho^2\de,\quad
|\Gamma^{\phi}_{T^{-1}z'_2}(T^{-1}z'_1,T^{-1}z'_2)|\sim \rho \, C_0^3\rho^2(1\vee \de),
$$
hence
$$
|\Gamma^{\phi_\de}_{z'_1}(z'_1,z'_2)|\sim
C_0^3(1\wedge \de), \quad
|\Gamma^{\phi_\de}_{z'_2}(z'_1,z'_2)|\sim
C_0^3.
$$
Moreover, by \eqref{gammaz}, if $z_1,\bar z_1\in U_1$ and $z_2,\bar z_2\in U_2,$ then
$$
|\Gamma^{\phi}_{\bar z_1}(z_1,z_2)-\Gamma^{\phi}_{z_1}(z_1,z_2)|=(y_2-y_1)^2|\bar
y_1-y_1|\lesssim C_0^2 \rho^2\rho (1\wedge \delta)\le C_0^2\rho^3\delta,
$$
and
$$
|\Gamma^{\phi}_{\bar z_2}(z_1,z_2)-\Gamma^{\phi}_{z_2}(z_1,z_2)|=(y_2-y_1)^2|\bar
y_2-y_2|\lesssim C_0^2 \rho^2\rho\le C_0^2 \rho^3(1\vee \delta).
$$
Hence, for $C_0$ sufficiently large,
$$
|\Gamma^{\phi}_{T^{-1}\bar z'_1}(T^{-1}z'_1,T^{-1}z'_2)|\sim  \, C_0^3\rho^3\de,\quad
|\Gamma^{\phi}_{T^{-1}\bar z'_2}(T^{-1}z'_1,T^{-1}z'_2)|\sim  \, C_0^3\rho^3(1\vee \de),
$$
hence
$$
|\Gamma^{\phi_\de}_{\bar z'_1}(z'_1,z'_2)|\sim
C_0^3(1\wedge \de), \quad
|\Gamma^{\phi_\de}_{\bar z'_2}(z'_1,z'_2)|\sim
C_0^3.
$$
This proves \eqref{transscaled}. For $\delta>1,$ $\phi_\delta$ does not depend on
$\delta,$
 and the claim in the  last  statement of the lemma is trivially verified.
\end{proof}

\setcounter{equation}{0}
\section{Statements of the bilinear estimates. The proofs.}\label{sect:bilin}

We are now in a position to  establish  the following  sharp bilinear Fourier extension
estimates for admissible pairs:

\begin{thmnr}\label{bilinear2}
Let $p>5/3,$ $q\ge2.$ Then, for every admissible pair  $(U_1,U_2)\in \cP^\de $
at scale $\de,$  the following bilinear estimates hold true:
  \begin{align*}
	 \|\ext_{U_1}(f)\ext_{U_2}(g)\|_p
	\leq C_{p,q} \delta^{2(1-1/p-1/q)}\rho^{6(1-1/p-1/q)}
\|f\|_q\|g\|_q,\quad\text{if}\quad \delta> 1,
 \end{align*}
  and
 \begin{align*}
	 \|\ext_{U_1}(f)\ext_{U_2}(g)\|_p
	\leq C_{p,q}\, \delta^{5-3/q-6/p}\rho^{6(1-1/p-1/q)}
\|f\|_q\|g\|_q,\quad\text{if}\quad
\delta\le 1,
 \end{align*}
 with constants that are independent of the given pair, of $\rho,$ and of $\de.$
\end{thmnr}

\begin{remnr}\label{remarkbilin}
Recall that for $\de>1$ the sets $U_1$ and $U_2$ are essentially rectangular boxes of
dimension $\rho^2\de\times \rho,$  and notice that our estimates  for this case do agree
with the ones given in Proposition 2.1 in \cite{v05} for the case of the saddle.
\end{remnr}

By the considerations in the previous section, Theorem \ref{bilinear2} reduces to the
following statement for the prototypical case.

\begin{thmnr}[prototypical case]\label{bilinear}
Let $p>5/3,$ $\delta\le1/10,$ and let $(U_1,U_2)$ be an  admissible pair given by   \eqref{U1},\eqref{U2}.
Then
\begin{align}\label{bilinest}
	 \|\ext_{U_1}(f_1),\ext_{U_2}(f_2)\|_p \leq C_p \, \delta^{\frac{7}{2}-\frac{6}{p}}
\|f_1\|_2\|f_2\|_2
 \end{align}
 for every $f_1\in L^2(U_1)$ and   $f_2\in L^2(U_2).$
\end{thmnr}

We begin by explaining in more detail here the reduction of Theorem \ref{bilinear2} to
Theorem \ref{bilinear}.  Subsection \ref{bilinarg}  will  then be devoted to the proof of
Theorem \ref{bilinear}.

\bigskip
\noindent {\it Reduction of Theorem \ref{bilinear2}  to Theorem \ref{bilinear}.} Fix $p>5/3$ and $q\ge 2,$  and
assume without loss of generality that
$U_1=U_1^{x_1^0,y_1^0,\delta}$ and $U_2= U_2^{t_2^0,y_1^0,y^0_2,\delta}$ form an admissible pair
of
type 1.

Assume first that $\delta\ge1/10.$ In this case, Lemma \ref{U'} shows that
 the conditions of Lee's Theorem 1.1 in \cite{lee05} are satisfied for the patches of
 surface  $S'_1$ and $S'_2$ which are the graphs of $\phi_\de$ over the sets $U'_1$ and
 $U'_2$  given by  \eqref{scaledsets1} and \eqref{scaledsets2}, and we can conclude that
 there is a constant $C'>0$ which does not depend on  $U'_1,U'_2$  such that $
 \|\ext^\de_{U'_1}(\tilde f)\ext^\de_{U'_2}(\tilde g)\|_p\le C_p\|\tilde f\|_2\|\tilde
 g\|_2.$  Here,  the  operators $\ext^\de_{U'_i}$ are essentially given by
 $$
\ext_{U'_i}^\de h(\xi)=\int_{U'_i} h(x',y') e^{-i(\xi_1 x'+\xi_2 y'+\xi_3\phi_\de(x',y'))}
\, dx' dy',\quad i=1,2.
$$
 Since $|U'_1|=1$ and $|U'_2|\sim 1,$  by H\"older's inequality this implies that
         \begin{align*}
    \|\ext^\de_{U'_1}(\tilde f)\ext^\de_{U'_2}(\tilde g)\|_p\le C_{p,q}\|\tilde
    f\|_q\|\tilde g\|_q.
     \end{align*}
 By undoing the change of coordinates and using \eqref{phi-phide}, this leads to   the
 estimates
$$
 \|\ext_{U_1}(f)\ext_{U_2}(g)\|_p
	\leq C_{p,q} \big[\rho^3 (1\vee \de)\big]^{2(1-1/p-1/q)} \|f\|_q\|g\|_q.
$$

Assume next that $\de\le 1/10 .$ Then Lemma \ref{U'} shows that $ U'_1=[0,\delta)^2,$ and
	$U'_2=\{(x',y'):0\le y'-b< 1,\,0\le x'+\frac{y'^2}{1\vee\delta}-a< \delta\},$ where
$\phi_\de (x',y')=\phi(x',y').$ Thus, in this situation we may reduce  to the prototype
situation studied in Theorem \ref{bilinear}, by means of a change of coordinates of the
form $(x',y')=(C_0^2 x'', C_0 y''),$  and thus obtain the estimate
$$
    \|\ext^\de_{U'_1}(\tilde f)\ext^\de_{U'_2}(\tilde g)\|_p\le C_{p}\de^{7/2-6/p}\|\tilde
    f\|_2\|\tilde g\|_2.
$$
Since here $|U'_1|=\de^2$ and $|U'_2|\sim \de,$ H\"older's inequality then implies that
	$$
    \|\ext^\de_{U'_1}(\tilde f)\ext^\de_{U'_2}(\tilde g)\|_p\le
    C_{p,q}\de^{3(1/2-1/q)}\de^{7/2-6/p}\|\tilde f\|_q\|\tilde g\|_q=
    C_{p,q}\de^{5-3/q-6/p}\|\tilde f\|_q\|\tilde g\|_q.
$$
By undoing the change of coordinates, we again pick up an extra factor   $\big[\rho^3
(1\vee \de)\big]^{2(1-1/p-1/q)}=\rho^{6(1-1/p-1/q)}$ and arrive at the  claimed estimate
for the case $\de\le {1/10}.$
\qed

\medskip
\subsection{The bilinear argument: proof of Theorem \ref{bilinear}}\label{bilinarg}
As the bilinear method is by now standard, we will only give a brief sketch of the proof
of Theorem \ref{bilinear}, pointing out the necessary modifications  compared to the
classical case of elliptic surfaces.

First, recall  from Subsection \ref{proto} that we had passed from our original coordinates $(x,y)$ to the coordinates $(\bar x, \bar y)$ by means of the scaling transformation $(x,y)=A(\bar x, \bar y)=(\delta\bar x,\bar y),$ and had put 
$\phi^s(\bar z):=\phi(A\bar z)/{\mathfrak a},$ with  $\mathfrak a:=\det A=\delta$.

We are going to establish the following  bilinear Fourier extension estimate for the
scaled patches of surface $S_i^s$ which where defined as the  graphs of $\phi^s$ over the sets $U^s_i, i=1,2:$ 

\begin{align}\label{bilinestscaled}
	 \|\ext_{U^s_1}(f_1)\,\ext_{U^s_2}(f_2)\|_p \leq C_p \,
\delta^{\frac{5}{2}-\frac{4}{p}} \|f_1\|_2\|f_2\|_2,
 \end{align}
 for every $f_1\in L^2(U^s_1)$ and every  $f_2\in L^2(U^s_2).$ Here, $\ext_{U^s_i}(f_i)$
 denotes the Fourier transform of $f_i d \sigma_i,$ where $\sigma_i$ is the pull-back of
 the  Lebesgue measure on $U^s_i$ to the surface $S^s_i$ by means of the projection onto
 the $(\bar x,\bar y)$- plane, i.e.,
 $$
\ext_{U^s_i}(f_i)(\xi,\tau)=\widehat {f_i d \sigma_i}(\xi,\tau)=\int_{U^s_i}f_i(\bar z)
e^{-i[\xi \bar z+\tau \phi^s(\bar z)]}\, d\bar z.
$$
 Scaling back to our original  coordinates, we obtain from this estimate that
 \begin{align*}
	 \|\ext_{U_1}(f_1)\, \ext_{U_2}(f_2)\|_p  \le&(\det A)^{1-1/p}{\mathfrak a}^{-1/p}
\delta^{5/2-4/p}
\|f_1\|_2\|f_2\|_2\\
	=& \delta^{7/2-6/p} \|f_1\|_2\|f_2\|_2,
\end{align*}
hence  \eqref{bilinest}.
\medskip

We start by recalling that a first important step in the bilinear argument  consists in a
wave packet decomposition of the functions $\widehat {f_i d \sigma_i}$ (compare, e.g.,
\cite{lee05} for details on wave packet decompositions).

For the construction of the wave packets at scale $R$ in the bilinear argument,  one first
decomposes the functions $f_i$ into well-localized (modulated) bump functions at scale
$1/R,$  localized near points $v,$ say $\varphi_v(\bar z)=\varphi(R(\bar z-v)),$
 and then considers
their Fourier extensions $\widehat{\varphi_vd\sigma}$ (here $\sigma$ denotes any of the
measures $\sigma_i$).
Then, by means of a Taylor expansion, one finds that  (essentially)
\begin{align*}
	\widehat{\varphi_vd\sigma}(\xi,\tau)=&\int \varphi(R\bar z)e^{-i[\xi(v+\bar z) +
\tau\phi^s(v+\bar z)]} \,d\bar z	 \\
	=\  & R^{-2}e^{-i\xi v}
	\int \varphi(\bar z)e^{-i[R^{-1}\xi\bar z+\tau\phi^s(v+R^{-1} \bar z)]}\, d\bar z\\
	= R^{-2}& e^{-i[\xi v + \tau\phi^s(v)]}
	\int \varphi(\bar z)e^{-i[R^{-1}(\xi+\tau\nabla\phi^s(v))\bar z +
\frac{1}{2}R^{-2}\tau
\bar z^tH\phi^s(v')\bar z]}\, d\bar z,
\end{align*}
where $v'$ is on the line segment between $v$ and
$v+R^{-1}\bar z.$
Following \cite{lee05}, Lemma 2.3, integration by parts shows that the wave packet is
then
associated to the region where the complete phase satisfies
$|(R^{-1}\xi\bar z+\tau\phi^s(v+R^{-1}\bar z))|\le c$ for every $\bar z$ with $|\bar
z|\leq
1,$ say with $c$ small,  on which  $\widehat{\psi_vd\sigma}$ is essentially constant.
This condition requires  in particular  that the  usual condition $|\xi+\tau
\nabla\phi^s(v)|\leq R$ holds true. Recall here also from  \eqref{bddgrad} that  for $v\in
U_1^s\cup U_2^s,$ we have $|\nabla\phi^s(v)|\lesssim1.$
Moreover, if we assume that  $v=(\bar x_v,\bar y_v)\in U_1^s$, then $|\bar y_v|\leq
c_0\delta$, so  that, if $R^{-1}\le\delta,$  by \eqref{hessphis},
$\|H\phi^s(v')\|\lesssim 1$.  Hence we obtain the usual condition
$|\tau|\leq R^2$.

Note also that the higher order derivatives of the phase, $R^{-1}\xi\bar
z+\tau\phi^s(v+R^{-1} \bar z),$ are bounded by constants.

This means  that the wave packets  associated to $f_1$ and the patch of hypersurface
$S_1^s$ are essentially supported in tubes  $T_1$ of the form
	\[ T_{1}=\{(\xi,\tau):|\xi+\tau\nabla\phi^s(v)|\leq R, |\tau|\leq R^2\},
\]
respectively ``horizontal''  translates in $\xi$ of them (due to modulations). Notice the
standard fact that $T_1$ is a tube of dimension $R\times R\times  R^2$ whose long axis is
pointing in the direction  of the normal vector $N(v):=(\nabla\phi^s(v),-1)$ to $S_1^s$ at
the point $(v,\phi^s(v)).$

\smallskip
However, if $v=(\bar x_v,\bar y_v)\in U_2^s$, then $|\bar y_v|\sim 1$, so  that, by
\eqref{hessphis}, if $R^{-1}\le\delta,$ $\|H\phi^s(v')\|\sim 1/\delta$. To
bound $R^{-2}|\tau \bar z^tH\phi(v')\bar z|\lesssim 1$,  we thus here  need to assume that
$|\tau|\leq \delta R^2,$    and the wave packets associated to $f_2$ and patch of
hypersurface  $S_2^s$  are  thus essentially supported in shorter tubes $T_2$ of the form
	\[ T_{2}=\{(\xi,\tau):|\xi+\tau\nabla\phi^s(v)|\leq R, |\tau|\leq\delta R^2\},
\]
respectively ``horizontal''  translates in $\xi$ of them.
\smallskip

The wave packets associated to such tubes will be denoted by $\phi_{T_i},\, i=1,2.$

\smallskip
There is a technical obstacle here to be noticed, which is of a similar nature as a
related
problem that had arisen   in \cite{bmv16}: to ensure that the wave packets are tubes, we
need that $\delta R^2>R$, i.e.,  $R>\delta^{-1}.$ For the usual induction on scales
argument  this creates the difficulty that we  cannot simply induct in the standard  way
on
the scales $R>1.$
Instead, we change variables $R'=R\delta,$ and induct on the scales $R'>1$. The wave
packets $T_2$ are of then of  dimension
$\frac{R'}{\delta}\times\frac{R'}{\delta}\times\frac{R'^2}{\delta},$ where now indeed
$\frac{R'^2}{\delta}>\frac{R'}{\delta}$.
\medskip

Following further on the bilinear method, we have to  consider localized  estimates  at
scale $R'$ of the form
 \begin{align}\label{inducthypo}
	 \|\ext_{U_1}(f_1)\, \ext_{U_2}(f_2)\|_{L^p(Q(R'))} \leq C(\delta) C_\alpha
(R')^\alpha
\|f_1\|_2\|f_2\|_2,
 \end{align}
where $Q(R')$ is a cuboid determined by the wave packets, and need to push down the
exponent $\alpha$ by means of induction on scales. Wave packet decompositions then allow
to
reduce  these estimates to bilinear estimates for the associated wave packets.

The corresponding  $L^1$-estimate is trivial (compare \cite{lee05} , or \cite{bmv16}, for
details): since the wave packets of a given type arising in these decompositions are
almost
orthogonal, one easily finds by Cauchy-Schwarz' inequality that
\begin{align}\nonumber
	\|\sum_{T_1,T_2}\phi_{T_1}\phi_{T_2}\|_{1} \leq
&\|\sum_{T_1}\phi_{T_1}\|_2\|\sum_{T_2}\phi_{T_2}\|_{2}\leq \Big(\frac{R'^2}{\delta^2}\#
T_1\Big)^{1/2} \Big(\frac{R'^2}{\delta} \# T_2\Big)^{1/2}\\
	=& \delta^{-3/2} R'^2(\# T_1\, \# T_2)^{1/2}. \label{bilin1}
\end{align}
As for further $L^p$-estimates, grossly oversimplifying, the bilinear method allows to
devise some ``bad'' subset of $Q(R')$ whose contributions are simply  controlled by means
of  the induction on scales hypothesis, and a ``good'' subset, on which we can obtain a
strong $L^2$- estimate by means of  sophisticated geometric-combinatorial considerations,
essentially of the form (without going into details)
\begin{align}\label{bilin2}
	\|\sum_{T_1,T_2}\phi_{T_1}\phi_{T_2}\|_{2} \leq R^{-1/2} (\# T_1\, \# T_2)^{1/2}
	= \delta^{1/2}R'^{-1/2} (\# T_1\, \# T_2)^{1/2}.
\end{align}
For $1<p<2,$ interpolation between these estimates in \eqref{bilin1} and \eqref{bilin2}
gives
 \begin{align*}
	\|\sum_{T_1,T_2}\phi_{T_1}\phi_{T_2}\|_{p}
	\leq& (\delta^{-3/2}R'^2)^{2/p-1}(\delta ^{1/2}R'^{-1/2})^{2-2/p}  (\# T_1\#
T_2)^{1/2}
\\
	=& \delta^{5/2-4/p} (R')^{5/p-3}  (\# T_1\# T_2)^{1/2}.
\end{align*}
Since $R'\ge 1,$ again  grossly simplifying, this (very formal) argument will in the end
show that  for $p>5/3$ indeed the bilinear estimate \eqref{bilinestscaled} holds true.

\medskip
In this very rough description of the bilinear approach  in our setting we have suppressed
a number of subtle  and important issues which we shall explain next  in some more detail.

\medskip
Indeed, the proof of the crucial $L^2$- estimate \eqref{bilin2} requires more careful
considerations.
 For the combinatorial argument to work, we not only  need the lower bounds on the
 transversalities  given by  Lemma \ref{transvscaled}, but also have to make sure that
 the
 tubes $T_i$  of a given type (on which the wave packets are essentially supported) are
 separated  as the base point varies along  the intersection curves at distances of order
 $1/R.$

  Let us explain this in more detail.  We take a collection of $1/R$ - separated points
  along the projection  onto the $\bar z=(\bar x,\bar y)$ - plane of an intersection curve
  $\Pi^s_{\bar z_1',\bar z_2'}$ associated to  the patches of surface $S_1^s$ and $S_2^s.$
  For each point $\bar z$  of this collection we consider the point $\bar z_1:=\bar z-\bar
  z_2'\in U_1^s.$ We fix another point $p_0\in\mathbb R^3$ and consider all the tubes
  $T_1$
  which are  associated to such  base points $\bar z_1$ and which  pass through the given
  point $p_0.$
What the geometric-combinatorial  argument then requires is that  the directions of these
tubes be separated so that the tubes
$T_1\cap B(p_0, R^2/2)^c$ have bounded overlap (see \cite{v05}, \cite{lee05},
\cite{bmv16}).

Given the dimensions of the tubes $T_1$ of type $1,$ it is clear that this kind of
separation is achieved if the directions of the normals to the re-scaled surface $S^s$ at
those points $(\bar z_1,\phi^s(\bar z_1))$ are $R/R^2=1/R$ -  separated.

Similarly,  given the dimensions of the tubes $T_2$ of type $2,$ we shall also need that
the  directions of the normals to the surface at the points $(\bar z_2,\phi(\bar z_2)),$
for $\bar z_2:=\bar z-\bar z_1',$ are $R/(\delta R^2)=(1/\delta)(1/ R)$  - separated. The
sizes of the entry $\partial^2 \phi^s/\partial \bar y^2$ of the Hessian of $\phi^s$  in
\eqref{hessphis} at the points of $U_1^s$  respectively  $U_2^s$ and the fact that the
tangents $\omega$ to the curve $\gamma$ are essentially diagonal (compare \eqref{tangent})
guaranty that the desired separation condition is indeed satisfied.

\medskip

Another obstacle, which again already arose in  \cite{bmv16}, consists in setting  up the
base case for the induction on scales argument, i.e., setting up a suitable estimate of
the
form \eqref{inducthypo}  for some initial, possibly very large value of
the exponent $\alpha.$

In the classical setting of elliptic surfaces, the  na\"\i f  and easily established
estimate  of the form
\begin{align*}
	 \|\ext_{U_1}(f_1)\, \ext_{U_2}(f_2)\|_{L^p(Q(R'))}  \leq |Q(R')|^{1/p}
\|f_1\|_1\|f_2\|_1
\end{align*}
would work, but in our setting, this would not  give  the right power of $\delta$ needed
to
establish \eqref{inducthypo}.

 We therefore follow our approach in \cite{bmv16}(compare  Lemma 2.10,  2.11), which
 provides us with  the  following a-priori $L^2$-estimate (not relying on the
 afore-mentioned geometric argument):
\begin{align*}
	\|\sum_{T_1,T_2}\phi_{T_1}\phi_{T_2}\|_{2} \leq R^{-1/2} \sup_{i=1,2}\sup_{\Pi}(\#
\cT_i^\Pi)^{1/2}(\# T_1\, \# T_2)^{1/2},
\end{align*}
where $ \cT_i^\Pi$ denotes the  set of all tubes of type $i$ associated to base points
along a given  intersection curve $\Pi$ of translates of $S_1^s$ and $S_2^s$. We are done
if we can show that $\# \cT_i^\Pi$ is bounded by some power of $R'$ but independently of
$\delta$.\\
We already saw in \eqref{tangent} that the tangent $\omega$ of the intersection curve is
essentially diagonal. After scaling, $U_2^s$ is a set of dimensions $1\times 1$, but
$U_1^s$ is a rectangle of dimensions $1\times\delta$, so an essentially  diagonal
intersection curve can have length at most $\Landau(\delta)$. Since the separation of the
base points of our wave packets  along this curve  is of size $1/R=\delta/R'$, we see
that
indeed we must have
$\# \cT_i^\Pi\leq C R'$.

This completes our sketch of proof of Theorem \ref{bilinear}.
\hfill $\Box$\\

After distributing our preprint though the ArXiv, we learned from Timothy Candy that the
bilinear estimate in Theorem \ref{bilinear} could  also be deduced from his more
general bilinear estimates in Theorem 1.4 of \cite{can17}, after applying the crucial
scaling in $x$ that we use in Subsection \ref{proto}. The convexity assumptions on the sets $\Lambda_j$
in his theorem is not really necessary, as he pointed out to us. We wish
to thank him for informing us about this.


\setcounter{equation}{0}
\section{The Whitney-type decomposition and its overlap}\label{whitn}

\subsection{The Whitney-type decomposition of $V_1\times V_2$}\label{whitneydecomp}
Let $(V_1,V_2)$ be an admissible pair of strips as defined in Subsection \ref{preciseadmissible}.
Recall the definition of admissible pairs  of sets from the same subsection, and that we had also introduced
there the sets  $\cP^\de$  respectively $\tilde \cP^\de$ of  admissible pairs of type
1 respectively   type 2 at scale
$\de,$ and by $\cP$ respectively $\tilde \cP$ we had denoted the  corresponding unions over all dyadic
scales $\de.$

\begin{lemnr}\label{covering}
The following covering and overlapping properties hold true:
\begin{itemize}
\item[(i)] For fixed dyadic scale $\de,$ the subsets $U_1\times U_2, \, (U_1,U_2)\in
    \cP^\de,$ of $V_1\times V_2\subset Q\times Q$ are pairwise disjoint,  as likewise
    are the subsets $\tilde U_1\times \tilde U_2, \, (\tilde U_1,\tilde U_2)\in
    \tilde\cP^\de.$

\item[(ii)] If $\de$ and $\de'$ are dyadic scales,  and if  $(U_1,U_2)\in \cP^\de$ and
    $(U'_1,U'_2)\in \cP^{\de'},$ then the sets $U_1\times U_2$ and $U'_1\times U'_2$
    can only intersect if   $\de/\de' \sim_{2^7} 1.$ In the latter case, there is only
    bounded overlap. I.e., there is a constant  $M\le2^6$  such that for
    every $(U_1,U_2)\in \cP^\de $  there are at most $M$  pairs   $(U'_1,U'_2)\in
    \cP^{\de'}$  such that
$(U_1\times U_2)\cap ( U'_1\times U'_2)\ne \emptyset,$  and vice versa.

The analogous statements apply to admissible pairs in $\tilde\cP.$

\item[(iii)]  If  $(U_1,U_2)\in \cP^\de$ and  $(\tilde U_1,\tilde U_2)\in
    \tilde\cP^{\de'},$ then $U_1\times U_2$ and $\tilde U_1\times \tilde U_2$ are
    disjoint too, except possibly when both $\de,\de'\ge 1/800$ and
    $\de\sim_{2^{10}}\de'.$ In the latter case, there is only bounded overlap. I.e.,
    there is a constant  $N=\Landau (C_0)$   such that for every $(U_1,U_2)\in \cP^\de
    $  there are at most $N$  pairs   $(\tilde U_1,\tilde U_2)\in \tilde\cP^{\de'}$
    such that
$(U_1\times U_2)\cap (\tilde U_1\times \tilde U_2)\ne \emptyset,$  and vice versa.

\item[(iv)]  The product sets associated to all admissible pairs cover $V_1\times V_2$
    up to a set of measure $0,$  i.e.,
$$
V_1\times V_2=\Big(\bigcup\limits_{(U_1,U_2)\in \cP}U_1\times U_2
\Big)\cup\Big(\bigcup\limits_{(\tilde U_1,\tilde U_2)\in \cP}\tilde U_1\times \tilde
U_2\Big)
$$
in measure.
\end{itemize}
\end{lemnr}

\begin{proof}
(i) Let $(U_1,U_2),(U'_1,U'_2)\in \cP^\de$ be different admissible pairs of type 1 at
 scale $\de.$ Then we shall see that already the sets $U_1$ and $U'_1$ will be disjoint.
 Indeed,  this will obviously be true if the corresponding  points $y_1^0$  are  different
 for $U_1$ and $U'_1,$ and  if they  do agree, then this will be true  because of
 different
 values of $x_1^0.$
 \medskip

 (ii) The first statement follows immediately from Lemma  \ref{sizeofdeltas}. To prove the second one, assume
 that $\de/\de' \sim_{2^7} 1,$ and that  $U_1=U_1^{x_1^0,y_1^0,\delta},
 U'_1=U_1^{(x')_1^0,(y')_1^0,\delta'}.$ Recall that $U_1=U_1^{x_1^0,y_1^0,\delta}$ is
 essentially a rectangular box of dimension $\sim \rho^2\de\times \rho(1\wedge \de),$
 containing the point $(x_1^0,y_1^0),$  whose longer side has slope $y_1^0$ with respect
 to
 the $y$-axis, and an analogous statement is true of $U'_1.$ If $U_1\cap U'_1\ne
 \emptyset,$ then we must have $|y_1^0-(y')_1^0|\le\rho(1\wedge \de),$ and so
 the  slopes for these two boxes can only differ of size  at $\rho(1\wedge \de),$ which
 implies that  both  boxes must be essentially of the same direction and dimension. This
 easily implies the claimed  overlapping properties, because  of the separation properties
 of the points in  $\I_1$ and $\I.$

 \medskip
 (iii) If either $\de<1/800,$ or $\de'<1/800,$  or if $\de\nsim_{2^{10}}\de',$ then  Lemma \ref{sizeofdeltas} and  Lemma \ref{sizeofdeltas2} show that $U_1\times U_2$ and  $\tilde U_1\times \tilde U_2$ must be disjoint. In
 the
 remaining case where both $\de,\de'\ge 1/800$ and  $\de\sim_{2^{10}}\de',$  assume that
 $U_1=U_1^{x_1^0,y_1^0,\delta}$ and
 $\widetilde{U}_1=\widetilde{U}_1^{x^0,y_2^0,\delta'}$  are so that  $U_1\cap \tilde
 U_1\ne
 \emptyset.$ Then observe that  both $U_1$ and $\tilde U_1$ are essentially rectangular
 boxes of dimension $\sim \rho^2\de\times \rho,$ where $U_1$ has slope $y_1^0$ with
 respect
 to the $y$-axis, and $\tilde U_1$ has slope $y_2^0$ with respect to the $y$-axis. Since
 $|y_1^0-y_2^0|\sim C_0 \rho\lesssim C_0 \rho\de,$ this  shows that $\tilde U_1$ must be
 contained within a rectangular box of dimension $\sim (C_0\rho^2\de)\times \rho$ around
 $U_1,$ so that  there are  at most $\Landau(C_0)$ sets $\tilde U_1$ of this type  which
 can intersect $U_1.$

\medskip
 (iv)   Let $(z_1,z_2)\in V_1\times V_2$. Without loss of generality, we may assume by
 symmetry that $|\tau_{z_1}(z_1,z_2)|\le  |\tau_{z_2}(z_1,z_2)|,$ i.e., that
$$
|x_2-x_1+y_2(y_2-y_1)|\leq|x_2-x_1+y_1(y_2-y_1)|.
$$
We shall then show that there is an admissible pair $(U_1,U_2)\in \cP$ such that
$(z_1,z_2)\in U_1\times U_2$ (in the other case, we would accordingly find an admissible
pair  of type 2 with this property). We shall also assume that $|\tau_{z_1}(z_1,z_2)|>0,$
since  the set of pairs $(z_1,z_2)$ with $\tau_{z_1}(z_1,z_2)=0$ forms a set of measure
$0.$

Then there is  a unique dyadic $0<\delta\lesssim\rho^{-2}$ such that
$$C_0^2\rho^2\de\le |\tau_{z_1}(z_1,z_2)|<2C_0^2\rho^2\de.
$$

Chose  $y_1^0\in \I_1$ such that $0\le y_1-y_1^0< \rho(1\wedge\de),$ and then
$x_1^0,t_2^0\in \I$ such that
\begin{align*}
	0\le x_1-x_1^0+y_1^0(y_1-y_1^0)< \rho^2\de \quad
	\text{ and } \quad 0\le x_2-t_2^0+y_2(y_2-y_1^0)< \rho^2\de.
\end{align*}
Define as in \eqref{pointsinU} $z_1^0:=(x_1^0,y_1^0)$ and $z_2^0=(t_2^0-y_2^0(y_2^0-y_1^0),
y_2^0).$
We observe that, as  in the proof of Lemma \ref{sizeofdeltas}, these estimates  imply  that
the estimates \eqref{tausize1} and   \eqref{tausize2}
remain valid. In particular, we immediately see that
$|\tau_{z_1^0}(z_1^0,z_2^0)|\sim_4 C_0^2\rho^2\de,$ so that
\eqref{admissible1} is satisfied.

\smallskip
As for condition \eqref{admissible2},  if $\de>8,$ by means of \eqref{TV2+TV3'} and
\eqref{yseparation}  we can estimate
$$
|\tau_{z_2^0}(z_1^0,z_2^0)|\ge
|\tau_{z_1^0}(z_1^0,z_2^0)|-C_0^2\rho^2\ge
C_0^2\rho^2(\de/4-1)\ge
C_0^2\rho^2(1\vee \de)/8.
$$
On the other hand, if $\de<1/32,$ then we may estimate
$$
|\tau_{z_2^0}(z_1^0,z_2^0)|\ge
C_0^2\rho^2/4 -|\tau_{z_1^0}(z_1^0,z_2^0)|\ge
C_0^2\rho^2(1/4-4\de)\ge C_0^2\rho^2(1\vee \de)/8.
$$
What remains is the case where $1/32\le \de\le 8.$  Here we use the estimate
$$
|\tau_{z_2}(z_1,z_2)|\ge  |\tau_{z_1}(z_1,z_2)|\ge C_0^2\rho^2\de\ge  C_0^2\rho^2/32,
$$
which by \eqref{tausize2} implies
$$
|\tau_{z_2^0}(z_1^0,z_2^0)|\ge
C_0^2\rho^2/32-6C_0 \rho^2(1\vee\de)\ge C_0^2\rho^2/32-48C_0 \rho^2\ge
C_0^2\rho^2/64\ge
C_0^2\rho^2(1\vee \de)/512,
$$
if we choose $C_0$ sufficiently large.

Moreover, note that we always have
$$
|\tau_{z_2^0}(z_1^0,z_2^0)|\le
|\tau_{z_1^0}(z_1^0,z_2^0)|+C_0^2\rho^2\le 5C_0^2\rho^2 (1\vee
\de).
$$
Thus we have also verified \eqref{admissible2}. Hence,
$(U_1^{x_1^0,y_1^0,\delta},U_2^{t^0_2,y_1^0,y^0_2,\delta})$ is an admissible pair of type 1 at
scale $\delta.$
\end{proof}

\medskip

\subsection{Handling the overlap in the Whitney-type decomposition of $V_1\times V_2$}

For  $r=0,\dots, 9,$ we define the subset  $\cP_r:=\bigcup_j \cP^{2^{10j+r}}$ of $\cP.$ To
these subsets of admissible pairs, we associate the subsets
\begin{align}\label{Ar}
A_r:=\bigcup\limits_{(U_1,U_2)\in \cP_r}U_1\times U_2, \qquad \tilde
A_r:=\bigcup\limits_{(\tilde U_1,\tilde U_2)\in \tilde \cP_r}\tilde U_1\times \tilde
U_2,\qquad (r=0,\dots, 9),
\end{align}
of $V_1\times V_2.$ Then Lemma \ref{covering} shows the following:
 \begin{itemize}
\item[(i)] The unions in \eqref{Ar} are disjoint unions of the sets $U_1\times U_2,$
    respectively $\tilde U_1\times \tilde U_2.$
\item[(ii)]  There is a fixed number $N\gg 1$ such that the following hold true:

 For given $r\ne r'$ and $(U_1,U_2)\in \cP_r,$ there are at most $N$ admissible pairs
 $(U'_1,U'_2)\in \cP_{r'} $ such that
 $(U_1\times U_2)\cap ( U_1'\times  U_2')\ne \emptyset,$  and vice versa. Similarly,
 for given $r,r'$ and  $(U_1,U_2)\in \cP_r,$ there are at most $N$ admissible pairs
 $(\tilde U_1,\tilde U_2)\in\tilde \cP_{r'} $ such that
 $(U_1\times U_2)\cap (\tilde U_1\times \tilde U_2)\ne \emptyset,$ and vice versa.
 \item[(iii)] \hskip 2cm $V_1\times V_2=\bigcup\limits_{r=0}^9(A_r\cup \tilde A_r).$
 \end{itemize}

 We shall make use of the following identity, which follows easily by induction on $m:$ If
 $B_1,\dots,B_m$ are subsets of a given set $X,$ then
 \begin{align*}
 \chi_{B_1\cup\cdots\cup B_m}=\sum_{\emptyset\ne J\subset \{1,\dots, m\}} (-1)^{\# J+1}
 \chi_{\bigcap_{j\in J}B_j}.
\end{align*}
 Applying this to (iii), we find that
 \begin{align}\label{chiunion}
 \chi_{V_1\times V_2}=\sum_{\emptyset\ne J,J'\subset \{0,\dots, 9\}}(-1)^{\#J+\#J'+1}
 \chi_{(\bigcap_{r\in J}A_r)\cap( \bigcap_{r'\in J'}\tilde A_{r'}) }.
\end{align}
This will allow to reduce our considerations to finite intersections of the set $A_r$ and
$\tilde A_{r'}.$ Indeed, let us define
$$
E(F)(\xi):=\int_{Q\times Q} F(z,z')  e^{-i[\xi\cdot(z,\phi(z))+\xi\cdot(z',\phi(z'))]}
\eta(z)\eta(z') \, dz dz',
$$
for any integrable function $F$ on $Q\times Q,$    so that in particular $E(f\otimes
g)=\ext(f) \ext(g).$ Then, by \eqref{chiunion}, if $f$ is supported in $V_1$ and $g$ in
$V_2,$
\begin{align}\label{chiunionE}
 \ext(f) \ext(g)=\sum_{\emptyset\ne J,J'\subset \{0,\dots, 9\}}(-1)^{\#J+\#J'+1}E\Big(
 (f\otimes g)\, \chi_{(\bigcap_{r\in J}A_r)\cap( \bigcap_{r'\in J'}\tilde A_{r'}))}\Big).
\end{align}

We may thus reduce considerations to restrictions of $f\otimes g$ to any of the
intersection
sets in \eqref{chiunionE}. So, let us  fix non empty subsets $J,J'\subset \{0,\dots, 9\}$
and  put
$$
B:= (\bigcap_{r\in J}A_r)\cap( \bigcap_{r'\in J'}\tilde A_{r'}).
$$
We then  choose $r_0\in J$ and note that $B\subset A_{r_0},$ where $A_{r_0}$ is the
disjoint union of the product sets $U_1\times U_2$ over all admissible pairs $(U_1,U_2)\in
\cP_{r_0},$ so that
\begin{align*}
B=\overset{\cdot}{\bigcup\limits_{(U_1,U_2)\in \cP_{r_0}}}B\cap (U_1\times U_2).
\end{align*}
And, by  (i),(ii),  each intersection
$B\cap (U_1\times U_2)$ is the finite disjoint union
\begin{align*}
B\cap (U_1\times U_2)=\overset{\cdot}{\bigcup\limits_{\iota}}\, \Omega_1^\iota\times
\Omega_2^\iota, \qquad (U_1,U_2)\in \cP_{r_0},
\end{align*}
of at most $N^{\#J+\#J'}$ measurable product subsets $\Omega_1^\iota\times
\Omega_2^\iota\subset U_1\times U_2.$

Since we shall   have to estimate the $L^p$- norm of $E((f\otimes g)\, \chi_B)$  later,
note that
$$
E((f\otimes g)\chi_{B\cap (U_1\times U_2)})=\sum\limits_\iota \ext(f\chi_{\Omega_1^\iota})
\ext(g\chi_{\Omega_2^\iota}).
$$
Thus, if we write our estimates in Corollary \ref{bilinear2} in the form
$ \|\ext_{U_1}(f)\ext_{U_2}(g)\|_p \leq C_{p,q}(U_1,U_2) $ \hfill $\times \|f\|_q\|g\|_q,$
since $\Omega_i^\iota\subset U_i,$ we see that also
$$
\|\ext(f\chi_{\Omega_1^\iota}) \ext(g\chi_{\Omega_2^\iota})\|_p \leq C_{p,q}(U_1,U_2)
\|f\chi_{U_1}\|_q\|g\chi_{U_2}\|_q,
$$
hence, for every  $(U_1,U_2)\in \cP_{r_0},$
\begin{align*}
\|E((f\otimes g)\chi_{B\cap (U_1\times U_2)}\|_p\leq  N^{\#J+\#J'} C_{p,q}(U_1,U_2)
\|f\chi_{U_1}\|_q\|g\chi_{U_2}\|_q.
\end{align*}
This shows that on the  set $B,$ we essentially get the same  estimates for the
contributions by subsets $U_1\times U_2$  that we get on $A_{r_0}.$  This will allow  us
to
reduce our considerations in the next section to the sets $A_r$ (respectively $\tilde
A_r$),  which are already disjoint unions of product set $U_1\times U_2$ associated to
admissible pairs.

 Observe finally that, for given $r,$
\begin{align}\label{ChiAr}
(f\otimes g)\chi_{A_r}=\sum\limits_{(U_1,U_2)\in \cP_{r}} (f\chi_{U_1}) \otimes
(f\chi_{U_2}).
\end{align}

\setcounter{equation}{0}
\section{Passage to linear restriction estimates  and proof of Theorem
\ref{mainresult}}\label{bilinlin}
To prove Theorem \ref{mainresult}, assume that $r>10/3$ and $1/q'>2/r,$ and put $p:=r/2,$
so that $p>5/3$, $1/q'>1/p.$ By interpolation with the trivial estimate for
$r=\infty,q=1,$
it is enough to prove the result  for $r$ close to $10/3$ and  $q$ close to 5/2, i.e.,
$p$
close to $5/3$ and $q$ close to 5/2. Hence, we may  assume that $p<2,$ $p<q<2p.$ Also, we
can assume that $\supp f\subset\{(x,y)\in Q:  y>0\}.$

We prove the linear estimates in two steps, following our steps in the construction in
Subsection \ref{pairs of sets}.
\smallskip

In a first step, we  fix  a scale $\rho$ and  shall prove  uniform bilinear Fourier
extension estimates  for admissible  pairs $V_1\backsim V_2$   of strips at scale $\rho$
(as defined in \eqref{Vi} of Subsection \ref{pairs of sets}).   Our goal will be to prove
the following

\begin{lemnr}\label{V1V2bilin}
 If  $V_1\backsim V_2$ form an  admissible pair of ``strips''
 $V_i=V_{j_i,\rho}=I_{j_i,\rho}\times [-1,1], \, i=1,2,$  at scale $\rho$ within
$Q,$  and if $f\in L^q(V_1)$ and $g\in L^q(V_2),$ then for the range of $p$'s and $q$'s
described above  we have
\begin{align}\label{VVbilin}
\|\ext_{V_1} (f) \ext_{V_2} (g)\|_p\lesssim  C_{p,q} \,\rho^{2(1-1/p-1/q)} \|f\|_q\,
\|g\|_q \  \text{for all} \ f\in L^q(V_1),  g\in L^q(V_2).
\end{align}
\end{lemnr}

\

We remark that, eventually, we shall choose $f=g$,  but for the arguments to follow it is
helpful to distinguish between  $f$ and $g$.

\medskip
 \begin{proof}
 To begin with, observe that by  means of an affine linear transformation  of the form
 \eqref{changeofvariables1},  we may ``move the strips $V_1, V_2$  vertically'' so that
 $j_1=0,$ which means  that $V_1$ contains the origin. {Also, by \eqref{admissibleV}, $C_0/8<|j_2|<C_0/2.$} This
 we shall assume throughout
 the
 proof. \smallskip

We  recall  from the previous section that it is essentially  sufficient  to consider
$E((f\otimes g)\chi_{A_r})$  in place of $\ext_{V_1} (f) \ext_{V_2} (g)$ (and similarly
for
$\tilde A_r$ in place of $A_r$). But then \eqref{ChiAr} shows that we may decompose
$(f\otimes g)\chi_{A_r}=\sum_\delta \sum_{i,i',j} f_{i,j}^\delta\otimes g_{i',j}^\delta$,
where $f_{i,j}^\delta:=f\chi_{U_{1}^{i\rho^2\de,j\rho(1\wedge \de),\delta}}$,
$g_{i',j}^\delta:=g\chi_{
U_{2}^{i'\rho^2\de,j\rho(1\wedge \de),{j_2\rho},\delta}},$ and where  each
$\big(U_{1}^{i\rho^2\de,j\rho(1\wedge \de),\delta},
U_{2}^{i'\rho^2\de,j\rho(1\wedge \de),{j_2\rho},\delta}\big)$ forms an
admissible pair,  i.e., \eqref{admissible1}, \eqref{admissible2} are satisfied. This
means
in particular, {by \eqref{admissible1}} that  $ |i-i'|\sim C_0^2.$ The summation  in $\de$ is here meant as
summation
over all dyadic $\de$ such that $\de\lesssim \rho^{-2}.$ We may and shall also assume that
$f$ and $g$ are supported on the set $\{y>0\}.$ Then
\begin{align}\label{EfgA}
E((f\otimes g)\chi_{A_r})
	=  \sum_{\delta{\ge1/10}}\sum_{i,i'} \widehat{f_{i}^\delta
d\sigma}\widehat{g_{i'}^\delta d\sigma}+ \sum_{{\delta< 1/10}}\sum_{i,i',j} \widehat{f_{i,j}^\delta
d\sigma}\widehat{g_{i',j}^\delta d\sigma}.
	\end{align}

The first sum (which collects all straight box terms)  can be treated by more classical  arguments (compare,
e.g., \cite{lee05} or  \cite{v05}), which in view of the first  estimate in Theorem \ref{bilinear2} then leads
to
a  bound for the contribution of that sum to $\|\ext_{V_1} (f) \ext_{V_2} (g)\|_p$ in \eqref{VVbilin} of the
order
$$
\sum_{1\lesssim \de\lesssim (\rho^2)^{-1}}C_{p,q}\, \,(\de\rho^3)^{2(1-\frac 1p-\frac 1q)} \|f\|_q\|g\|_q
\lesssim\rho^{2(1-\frac 1p-\frac 1q)} \|f\|_q\|g\|_q,
$$
as required. We  leave the details  to the interested reader.

We  shall therefore concentrate on the second sum in \eqref{EfgA}  (which collects all {curved} box terms)  in
which the admissibility conditions reduce to $|i-i'|\sim C_0^2.$

\smallskip

Let us then  fix $\delta\le1/10,$ and simplify notation by writing  $f_{i,j}:=f_{i,j}^\delta,$
 $g_{i,j}:=g_{i,j}^\delta,$ and $U_{1,i,j}:=U_{1}^{i\rho^2\de,j\rho(1\wedge \de),\delta}$,
 $
 U_{2,i',j}:=
 U_{2}^{i'\rho^2\de,j\rho(1\wedge \de),j_2\rho,\delta}$.

 As a first step in proving estimate \eqref{VVbilin}, we exploit some almost orthogonality
 with respect to the $x$-coordinate:

 \begin{lemnr}\label{lpsquare} For $1\le p\le 2,$  we have
\begin{align}\label{psquarefunc}
	\big\|\sum_{j,\,i,|i-i'|\sim C_0^2}\, \widehat{f_{i,j}d\sigma}\,
\widehat{g_{i',j}d\sigma}\big\|_{p}^p
	\lesssim \sum_{N=0}^{\rho^{-2}}\Big\|\sum_{i\in[N\de^{-1},(N+1)\de^{-1}]\, ,\atop
|i-i'|\sim C_0^2,\,j}  \widehat{f_{i,j}d\sigma}\,\widehat{g_{i',j}d\sigma}\Big\|_{p}^p.
\end{align}
\end{lemnr}
\noindent {\it Proof of Lemma \ref{lpsquare}:} Assume that $
i\in[N\de^{-1},(N+1)\de^{-1}],$ and that $z_1=(x_1,y_1)\in U_{1,i,j}$ and
$z_2=(x_2,y_2)\in U_{2,i',j},$ where $|i-i'|\sim C_0^2,$ so that  $(U_{1,i,j},
U_{2,i',j})\in \cP^\de$ is an admissible pair. Then, by Lemma \ref{covering} (a), we have
$$
|x_2-x_1+y_2(y_2-y_1)|\sim_8 C_0^2\rho^2\de,
$$
where $|y_2||y_2-y_1|\le C_0\rho\, C_0 \rho$  by \eqref{yseparation}, so that we may
assume
that
$
|x_2-x_1|\le 9\, C_0^2\rho^2.
$
This implies that $x_1+x_2=2x_1+\Landau(\rho^2),$ and thus
$x_1+x_2=2N\rho^2+\Landau(\rho^2),$ where the constant in the error term is of order
$C_0^2,$  hence
$$
U_{1,i,j}+U_{2,i',j}\subset [2N\rho^2-10\, C_0^2\rho^2,2N\rho^2+10\,C_0^2\rho^2]\times
[0,2C_0\rho].
$$
Notice that the family of intervals $\big\{[2N\rho^2-10\,
C_0^2\rho^2,2N\rho^2+10\,C_0^2\rho^2]\big\}_{N=0}^{\rho^{-2}}$ is almost pairwise
disjoint.
Therefore we may argue as in the proof of Lemma 6.1 in  \cite{TVV} in order to derive the
desired estimate.\hfill $\Box$

\smallskip

For our next step, recall that $U_{1,i,j}$ is  essentially a rectangular box of dimension
$\sim
\rho^2\de\times \rho\de,$ and  $U_{2,i',j}$ is a thin curved box of width $\sim\rho^2 \de$
and length $\sim\rho,$  contained in a rectangle of dimension $\sim \rho^2\times \rho$
whose axes are parallel to the coordinate axes. Notice also  that if $\rho=1,$ then
$U_{1,i,j}$ is  essentially a square of size
 $\de\times\de,$  whereas   $U_{2,i',j}$ is a thin curved box of width $ \de$ and length
 $\sim 1.$  In this special case, it will be useful to further decompose $U_{2,i',j}$ into
 squares
 of  size $\de\times\de.$

 Analogously, since we may pass from the case $\rho=1$ to the case of general $\rho$ in
 our
 definition  of admissible pairs
 by means of the dilations   $D_\rho(x,y):=(\rho^2 x,\rho y),$  for arbitrary $\rho$ we
 decompose  $U_{2,i',j}$
 in the $y$-coordinate into intervals of length $\rho\de,$ by putting
 $U^k_{2,i',j}:=\{(x,y)\in U_{2,i',j}: 0\le y-k\rho\de< \delta\rho\}.$
  Then
  \begin{align*}
 U_{2,i',j}=\overset{\cdot}{\bigcup\limits_k} \,U^k_{2,i',j},
 \end{align*}
 where  the union is over a set of $\Landau (1/\de)$ indices $k.$  Accordingly, we
 decompose
 $g_{i',j}=\sum_k g_{i',j}^k,$ where
 $g_{i',j}^k:=g\chi_{U^k_{2,i',j}}.$
Then we have the following uniform square function estimate:
\begin{lemnr}\label{squaref} For $1<p\le 2$ there exists a constant $C_p>0$ such that for
every  $N=0,\dots,\rho^{-2}$ we have
\begin{align}\label{squarefunc}
\Big\|\sum_{i\in[N\de^{-1},(N+1)\de^{-1}], \atop  |i-i'|\sim C_0^2,\,j}
\widehat{f_{i,j}d\sigma}\,
\widehat{g_{i',j}d\sigma}\Big\|_{p}
	\le C_p \Big\|\Big(\sum_{i\in[N\de^{-1},(N+1)\de^{-1}],\atop |i-i'|\sim C_0^2,\,j\,
,k}
|\widehat{f_{i,j}d\sigma}\,
\widehat{g_{i',j}^kd\sigma}|^2\Big)^{1/2}\Big\|_{p}.
\end{align}
\end{lemnr}
\noindent {\it Proof of Lemma \ref{squaref}:} Notice first that a translation in  $x$ by
$N\rho^2$ allows to reduce to the case $N=0,$  which we shall thus assume. Then the
relevant sets $U_{1,i,j}$ and $U_{2,i',j}$  will all have their $x$-coordinates in the
interval $[0,\rho^2].$

 For $i,\,i',\,j,\,k$ as above, set
$S_{1,i,j}:=\{(\xi,\phi(\xi)): \xi\in U_{1,i,j}\},$
$S_{2,i',j}^k:=\{(\xi,\phi(\xi)):\xi\in
U_{2,i',j}^k\}.$
  The key to the square function estimate  \eqref{squarefunc} is the following almost
  orthogonality lemma:

\begin{lemnr}\label{ao}
Assume $N=0,$ and denote by $\tilde D_\rho, \rho >0,$  the dilations on $\Bbb R^3$ given
by
$\tilde D_\rho(x,y,z):=(\rho^2 x, \rho y, \rho^3 z).$  Then there is a family of cubes
$\{Q_{i,i',j}^k\}_{i\in[0,\de^{-1}], |i-i'|\sim C_0\,,j\, ,k}$  in $\mathbb R^3$ with
bounded overlap, whose  sides  are parallel to the coordinate axes  and of  length
$\sim\delta,$     such that
$S_{1,i,j}+S_{2,i',j}^k\subset
\tilde D_\rho( Q_{i,i',j}^k).$
 \end{lemnr}

\noindent {\it Proof of Lemma \ref{ao}:}  Recalling the parabolic scalings
$D_\rho(x,y):=(\rho^2 x,\rho y)$ (under which
 the phase $\phi$ is homogeneous of degree 3),  we may apply the scaling  by $\tilde
 D_{\rho^{-1}}$ in $\Bbb R^3$ order to  reduce our considerations  to the  case $\rho=1.$

Then, as we have already seen,  $S_{1,i,j}$ and $S_{2,i',j}^k$ are contained in boxes of
side
length, say,  $2\delta$ and sides parallel to the axes, whose projections to the $x$-axis
lie within the unit interval $[0,1].$  Therefore we can choose for
$Q_{i,i',j}^k$  a square of of side length $4\delta,$  with  sides parallel to the axes,
with the property that $S_{1,i,j}+S_{2,i',j}^k\subset Q_{i,i',j}^k.$
We need to prove that the overlap is bounded.

Note that, if $(x_1,y_1)\in U_{1,i,j}$ and $(x_2,y_2)\in  U_{2,i',j}^k$ with $|i-i'|\sim
C^2_0,$ then, by Lemma \ref{sizeofdeltas}, since $\rho=1$,
 we have
\begin{align*}
|x_2-x_1+y_2(y_2-y_1)|=|\tau_{z_1}(z_1,z_2)|\sim C_0^2 \delta.
\end{align*}

It suffices to prove the following: if $(x_1,y_1),(x_2,y_2)$ and $(x_1',y_1'),(x_2',y_2')$
are so that each  coordinate of these points is bounded by a large multiple of $C_0,$ the
$y$-coordinates are positive and  satisfy $y_2-y_1\gtrsim C_0 $  (by the $y$-separation
\eqref{yseparation}), and
\begin{eqnarray*}
x_2-x_1+y_2(y_2-y_1)&\sim&C_0^2\de,\\
x_2'-x_1'+y_2'(y_2'-y_1')&\sim&C_0^2\de,\\
x_1+x_2&=&x_1'+x_2'+\Landau(\delta),\\
y_1+y_2&=&y_1'+y_2'+\Landau(\delta),\\
x_1y_1+\frac{y_1^3}3+x_2y_2+\frac{y_2^3}3&=&x_1'y_1'+\frac{(y_1')^3}3+x_2'y_2'+\frac{(y_2')^3}3
+\Landau(\delta),
\end{eqnarray*}
then
 \begin{align}\label{overlap2}
x_1'=x_1+\Landau(\delta), \ y_1'=y_1+\Landau(\delta),\ x_2'=x_2+\Landau(\delta),\
y_2'=y_2+\Landau(\delta).
\end{align}
Set
$$
a:=x_1+x_2, \quad b:=y_1+y_2,\qquad a':=x'_1+x'_2, \quad b':=y'_1+y'_2,
$$
and
$$
t_1:=x_1y_1+\frac{y_1^3}3,\qquad t_2:=x_2y_2+\frac{y_2^3}3.
$$
The analogous quantities defined by $(x_1',y_1'),(x_2',y_2')$ are denoted by $t'_1$ and
$t'_2.$
Notice that by our assumptions,  $a$ and $b$ only vary of order $\Landau (\de)$ if we
replace $(x_1,y_1),(x_2,y_2)$ by $(x_1',y_1'),(x_2',y_2').$
Then,
$$
t_1+t_2=2x_1y_1-bx_1-ay_1+ab+\frac{b^3}3-b^2y_1+by_1^2.
$$

We choose $c$ with $|c|\sim C^2_0,$ such that $x_2-x_1+y_2(y_2-y_1)=c\delta.$ Then we may
re-write $x_1=\big (a-c\delta+(b-y_1)(b-2y_1)\big)/2.$
Therefore,
$$
t_1+t_2=C(a,b)+\frac{bc\delta}2+y_1[-c\delta+\frac32
b^2]-3by_1^2+2y_1^3:=C(a,b)+\Landau(\delta)+\psi(y_1),
$$
where  $C(a,b)$ is a polynomial in $a,b$ and where we have put   $\psi(y):=\frac32
b^2y-3by^2+2y^3.$ Similarly, $t'_1+t'_2=C(a',b')+\Landau(\delta)+\psi(y'_1),$
and thus, since  $a=a'+\Landau(\delta), b=b'+\Landau(\delta),$ hence
$C(a,b)=C(a',b')+\Landau(\delta),$ we have
$$
\psi(y_1)=\psi(y'_1)+ \Landau(\delta).
$$
We may assume without loss of generality that $y_2-y_1>0$ (the other case can be treated
in
a similar way). Then, because of the $y$-separation \eqref{yseparation}, we have
$y_2-y_1\gtrsim C_0 $ and $y_1\ge 0$ and $b=y_2+y_1,$ so that  $y_1\le \frac{b-1}2$ and
$b\ge1.$ It is a calculus exercise to prove that in this situation, for $y\le\frac{b-1}2,$
we have $\psi'(y)\ge3/2.$
This shows that we must have $y_1'=y_1+\Landau(\delta),$ hence also
$y_2'=y_2+\Landau(\delta),$ and then our first two assumptions imply also the remaining
assertions in \eqref{overlap2}.

This finishes the proof of the almost orthogonality Lemma \ref{ao}.
\hfill $\Box$

\smallskip
To complete the proof of Lemma \ref{squaref}, define the linear operators
$$\widehat{T_{i,i',j}^kh}(\xi):=\chi_{\tilde D_\rho(Q_{i,i',j}^k)}(\xi)\hat h(\xi),
$$
 and
$$
S(h):=\{T_{i,i',j}^k h\}_{i\in [0,\de^{-1}]\,,|i-i'|\sim C^2_0,\,j,k},
$$ for $h\in L^{p'}(\mathbb R^3).$  By Lemma \ref{ao} and Rubio de Francia's estimate
\cite{rdf}, we have
$$
\Big\|\Big(\sum_{i\in [0,\de^{-1}],\,|i-i'|\sim
C_0^2,\,j,k}|T^k_{{i,i',j}}h|^2\Big)^{1/2}\Big\|_{p'}\le C
\|h\|_{L^{p'}(\mathbb R^3)},
$$
for $2\le p'<\infty$ (this is clearly true when $\rho=1,$ and the linear change of
coordinates given by $\tilde D_\rho$ does not change this estimate).  Then, by duality, we
have
the adjoint operator estimate
$$
\Big\|\sum_{i\in [0,\de^{-1}],\,|i-i'|\sim
C_0^2,\,j,k}T^k_{i,i',j}F_{i,i',j}^k\Big\|_{L^p(\mathbb R^3)}\le
C\Big\|\Big(\sum_{i\in [0,\de^{-1}],\,|i-i'|\sim
C_0^2,\,j,k}|F_{i,i',j}^k|^2\Big)^{1/2}\Big\|_{L^p(\mathbb
R^3)}.
$$
The estimate \eqref{squarefunc} follows by applying the estimate above to the family of
functions
$F_{i,i',j}^k:= \widehat{f_{i,j}d\sigma}\,\widehat{g_{i',j}^k d\sigma},$  since
$F_{i,i',j}^k$  is the Fourier transform of a function supported in  $\tilde D_\rho(
Q_{i,i',j}^k).$

\hfill $\Box$
\smallskip

Since $p/2<1$ in Lemma \ref{V1V2bilin},  from \eqref{psquarefunc} and  \eqref{squarefunc}
we deduce that
\begin{align*}
\big\|\sum_{j,\,i,|i-i'|\sim C_0^2}\, \widehat{f_{i,j}d\sigma}\,
\widehat{g_{i',j}d\sigma}\big\|_{p}^p
&\lesssim\sum_{N=0}^{\rho^{-2}} \Big\|\Big(\sum_{j, \, i\in[N\de^{-1},(N+1)\de^{-1}],\atop
|i-i'|\sim C_0^2}\sum_k
|\widehat{f_{i,j}d\sigma}\,\widehat{g_{i',j}^kd\sigma}|^2\Big)^{1/2}\Big\|_{p}^p\\
&\lesssim \sum_{N=0}^{\rho^{-2}}\sum_{j, \,i\in[N\de^{-1},(N+1)\de^{-1}],\atop |i-i'|\sim
C_0^2}\Big\|\Big(\sum_{k}
|\widehat{f_{i,j}d\sigma}\,\widehat{g_{i',j}^kd\sigma}|^2\Big)^{1/2}\Big\|_{p}^p\\
&=\sum_{j,\, i, |i-i'|\sim C_0^2}\Big\|\Big(\sum_{k}
|\widehat{f_{i,j}d\sigma}\,\widehat{g_{i',j}^kd\sigma}|^2\Big)^{1/2}\Big\|_{p}^p.
\end{align*}
Using Khintchine's inequality, it suffices to bound
\begin{align*}
	 \sum_{j, \,i,|i-i'|\sim C_0^2}\|\widehat{f_{i,j}d\sigma}\,\widehat{\tilde
g_{i',j}d\sigma}\|_{p}^p
\end{align*}
for all $\tilde g_{i',j} = \sum_k \epsilon_k g_{i',j}^k$, $\epsilon_k=\pm 1$. Note that
$\|\tilde g_{i',j}\|_q=\| g_{i',j}\|_q.$ Now, by Theorem \ref{bilinear2},
\begin{align*}
	\big(\sum_{j, \,i,|i-i'|\sim C_0^2} \|\widehat{f_{i,j}d\sigma}\widehat{\tilde
g_{i',j}d\sigma}\big\|_{p}^p\big)^{1/p}
	\lesssim& \delta^{5-3/q-6/p} \rho^{6(1-1/p-1/q)}\left(\sum_{j,\, i,|i-i'|\sim C_0^2}
\big\|f_{i,j}\big\|_q^p\big\|\tilde g_{i',j}\big\|_q^p\right)^{1/p}\\
=& \delta^{5-3/q-6/p} \rho^{6(1-1/p-1/q)}\left(\sum_{j,\, i,|i-i'|\sim C_0^2}
\big\|f_{i,j}\big\|_q^p\big\| g_{i',j}\big\|_q^p\right)^{1/p}.
\end{align*}

 Assume for a moment that  $j$ is fixed.   Then  notice  the following obvious, but
 crucial
 facts:
  \begin{itemize}
\item[(i)] For every $i,$ $U_{1,i,j}$ is contained in the horizontal   ``strip''
    $V_{j,\rho} $  given by  $0\le y-j\de\rho<\rho\de.$
\item[(ii)] The sets $U_{1,i,j}$ are mutually disjoint, and  that the same holds true
    for the sets $U_{2,i',j}.$
\end{itemize}

Let us therefore put $f_j:=f\chi_{V_{j,\rho}}.$   Then,  by Cauchy-Schwarz' inequality,
\begin{align*}
\sum_{i,|i-i'|\sim C^2_0} \big\|f_{i,j}\big\|_q^p\big\|g_{i',j}\big\|_q^p
	\lesssim& \left(\sum_i \big\|f_{i,j}\big\|_q^{2p}\right)^{1/2}
	\left(\sum_{i'}\big\|g_{i',j}\big\|_q^{2p}\right)^{1/2}\\
	\lesssim& \left(\sum_i \big\|f_{i,j}\big\|_q^{q}\right)^{p/q}
	\left(\sum_{i'}\big\|g_{i',j}\big\|_q^{q}\right)^{p/q}\\
	\lesssim& \big\|f_{j}\big\|^p_q\big\|g\big\|^p_q,
\end{align*}
where we have also used that  $2p\geq q.$ Consequently, since the total number of $j$'s
over which we are summing is of order $\Landau(\rho/\rho\de)=\Landau(1/\de),$  we get
\begin{align*}
	\big\|\sum_{i,|i-i'|\sim C^2_0,\, j} \widehat{f_{i,j}d\sigma}\,
\widehat{g_{i',j}d\sigma}\big\|_{p}
	\lesssim& \delta^{5-3/q-6/p} \rho^{6(1-1/p-1/q)}\big(\sum_{i, |i-i'|\sim C^2_0,\,j}
\big\|f_{i,j}\big\|_q^p\big\| g_{i',j}\big\|_q^p\big)^{1/p}\\
\lesssim& \delta^{5-3/q-6/p} \rho^{6(1-1/p-1/q)}\big(\sum_j\|f_j\|_q^p\big)^{1/p}\|g\|_q\\
\lesssim& \delta^{5-3/q-6/p}
\rho^{6(1-1/p-1/q)}(1/\de)^{1/p-1/q}\big(\sum_j\|f_j\|_q^q\big)^{1/q}\|g\|_q\\
\lesssim& \delta^{5-2/q-7/p} \rho^{6(1-1/p-1/q)}\|f\|_q\|g\|_q.
\end{align*}
Now observe that the exponent in our dyadic parameter $\delta$ is
\begin{align*}
	5-2/q-7/p = 3+2/q'-7/p \geq 3-5/p > 0,
\end{align*}
so that we can sum in \eqref{ChiAr}  over all dyadic scales $\delta\ll 1$ and arrive at
the
estimate  \eqref{VVbilin}.

\end{proof}
 Note that, up to this point,  we have been  working with a fixed admissible pair if strips $V_1\sim V_2$ in which 
 all functions where supported.
\smallskip

In a final step, we shall prove the linear Fourier extension estimate of Theorem
\ref{mainresult}
by summing over  the contributions by all admissible  pairs of strips $V_1\backsim V_2,$
over all scales $\rho.$  To this end, let us  write $f_{j,\rho}:=f\chi_{V_{j,\rho}}.$  In
view of our Whitney-decomposition \eqref{whitney1}, we may therefore write
$$
\widehat{f\,d\sigma}\, \widehat{f\,d\sigma}=\sum_{\rho\lesssim 1}\sum_{j, j\backsim
j'}\widehat{f_{j,\rho}\,d\sigma}\widehat{f_{j',\rho}\,d\sigma},
$$
where summation in $\rho$ is meant to be over all dyadic $0<\rho\lesssim 1$ and we wrote
$j\backsim j' $ to denote
$V_{j,\rho}\backsim V_{j',\rho}.$
Assume that $r=2p$  and $q$ satisfy the hypotheses of Theorem \ref{mainresult}, so that
$p>5/3$ and $1/q'>1/p.$
Then, by Lemma \ref{V1V2bilin}, for fixed $j,\,j',$
$$
\|\widehat{f_{j,\rho}\,d\sigma}\,\widehat{f_{j',\rho}\,d\sigma}\|_{L^p}\le C_{p,q}\,
\rho^{2(1-1/p-1/q)}\|f_{j,\rho}\|_q\|f_{j',\rho}\|_q.
$$
Therefore, since $p<2,$ using Minkowski's inequality and Lemma 6.1 in \cite{TVV}
\begin{align*}
\|\widehat{f\,d\sigma}\|_{2p}^{2}=\|\widehat{f\,d\sigma}\,\widehat{f\,d\sigma}\|_p
&=\big\|\sum_{\rho\lesssim 1}\sum_{j, j\backsim
j'}\widehat{f_{j,\rho}\,d\sigma}\widehat{f_{j',\rho}\,d\sigma}\big\|_p\\
&\le\sum_{\rho\lesssim 1} \big \| \sum_{j,j\backsim
j'}\widehat{f_{j,\rho}\,d\sigma}\widehat{f_{j',\rho}\,d\sigma}\big\|_p\\
&\le\sum_{\rho\lesssim 1} \bigg(\sum_{j,j\backsim j'}\big \|
\widehat{f_{j,\rho}\,d\sigma}\widehat{f_{j',\rho}\,d\sigma}\big\|_p^p\bigg)^{1/p}.
\end{align*}
For fixed $\rho,$  by Lemma \ref{V1V2bilin}
and Cauchy-Schwarz' inequality,
\begin{align*}
\sum_{j,j\backsim j'}\big \|
\widehat{f_{j,\rho}\,d\sigma}\widehat{f_{j',\rho}\,d\sigma}\big\|_p^p
&\lesssim  \rho^{2(1-1/p-1/q)p}\sum_{j,j\backsim j'}\|f_{j,\rho}\|_q^p\,
\|f_{j',\rho}\|_q^p\\
&\lesssim \rho^{2(1-1/p-1/q)p}\bigg(\sum_j\|f_{j,\rho}\|_q^{2p}\bigg)^{1/{2}}
\bigg(\sum_{j'}\|f_{j',\rho}\|_q^{2p}\bigg)^{1/{2}}\\
&\lesssim \rho^{2(1-1/p-1/q)p}\bigg(\sum_j\|f_{j,\rho}\|_q^{q}\bigg)^{p/{q}}
\bigg(\sum_{j'}\|f_{j',\rho}\|_q^{q}\bigg)^{p/{q}}\\
&\lesssim \rho^{2(1-1/p-1/q)p}\|f\|_q^p\, \|f\|_q^p.
\end{align*}
where we have again used that  $2p\geq q.$
Therefore, since  $1-1/p-1/q>0,$
$$
\|\widehat{f\,d\sigma}\|_{2p}^{2}
\lesssim\|f\|_q^2,
$$

This completes the proof of our Fourier extension (respectively restriction)  Theorem
\ref{mainresult}.

\hfill$\Box$


\thispagestyle{empty}

\renewcommand{\refname}{References}


\begin{thebibliography}{----------}

\bibitem [Be16]{be16}  Bejenaru, I.,  Optimal bilinear restriction estimates for general
    hypersurfaces and the role of the shape operator. Int. Math. Res. Not. IMRN  (2017),  no. 23, 7109--7147.
    
 \bibitem [BCT06] {BCT} Bennet, J.,  Carbery, A.,   Tao, T.,  On the multilinear
    restriction  and Kakeya conjectures. Acta Math.  196  (2006),  no. 2, 261--302.
    
\bibitem [Bo91] {Bo1} Bourgain, J., Besicovitch-type maximal operators and applications to
    Fourier analysis.  Geom. Funct. Anal. 22 (1991), 147--187.
\bibitem [Bo95a] {Bo2}  Bourgain, J.,  Some new estimates on oscillatory integrals. Essays
    in Fourier Analysis in honor of E. M. Stein. Princeton Math. Ser. 42, Princeton
    University Press, Princeton, NJ 1995, 83--112.
\bibitem [Bo95b]  {Bo3}  Bourgain, J.,  Estimates for cone multipliers. Oper. Theory Adv. Appl. 77 (1995), 1--16.
\bibitem [BoG11] {BoG}  Bourgain, J.,  Guth, L., Bounds on oscillatory integral operators
    based on multilinear estimates. Geom. Funct. Anal., Vol.21 (2011)
    1239--1295.
\bibitem [BMV16] {bmv16}  Buschenhenke, S., M\"uller, D.,  Vargas,  A., A Fourier
    restriction theorem for a two-dimensional surface of finite type. Anal. PDE 10-4 (2017), 817--891.
  \bibitem [BMV18] {bmv19}  Buschenhenke, S., M\"uller, D.,  Vargas,  A.,
  On  Fourier restriction for finite-type  perturbations of the  hyperboloid.
  preprint 2019,  arXiv:1902.05442.    
 \bibitem[Can17]{can17} Candy, T., Multi-scale bilinear restriction estimates for general phases.  Math. Annalen, https://doi.org/10.1007/s00208-019-01841-4.
\bibitem[Car67]{c67} Carleson, L., On the Littlewood-Paley theorem. Report,
    Mittag-Leffler Inst., Djursholm, 1967.
\bibitem[ChL17]{chl17} Cho, C.-H.,  Lee, J.,  Improved restriction estimate for hyperbolic
    surfaces in $\R^3$. J. Funct. Anal.  273  (2017),  no. 3, 917--945.
\bibitem[Co81]{co81} C\'ordoba, A., Some remarks on the Littlewood-Paley theory. Rend.
    Circ. Mat. Palermo Ser II 1(1981), Supplemento, 75-80.
\bibitem [F70] {F1}  Fefferman, C.,  Inequalities for strongly singular convolution
    operators. Acta Math., (1970), 9--36.
\bibitem[Gu16]{Gu16}  Guth, L.  A restriction estimate using polynomial partitioning. J.
    Amer. Math. Soc.  29  (2016),  no. 2, 371--413.
\bibitem[Gu17]{Gu17}  Guth, L.,   Restriction estimates using polynomial partitioning
    II. Acta Math. Vol. 221, No. 1 (2016), 81--142.
\bibitem [Gr81] {Gr} Greenleaf, A., Principal Curvature and Harmonic Analysis. Indiana
    Univ. Math. J. Vol. 30, No. 4 (1981).
\bibitem [IKM10] {ikm}  Ikromov, I. A.,  Kempe, M.,  M\"uller, D.,  Estimates for maximal
    functions associated with hypersurfaces in $\R^3$ and related problems in harmonic
    analysis. Acta Math. 204 (2010), 151--271.
\bibitem[IM11]{IM-uniform}  Ikromov, I. A.,   M\"uller, D.,  Uniform estimates for the
    Fourier transform of surface carried measures in  $\R^3$ and an application to Fourier
    restriction. J. Fourier Anal. Appl., 17 (2011), no. 6, 1292--1332.
\bibitem[IM15] {IM} Ikromov, I. A.,   M\"uller, D., Fourier restriction for hypersurfaces
    in three dimensions and Newton polyhedra. Annals of Mathematics Studies, 194.
    Princeton
    University Press, Princeton, NJ, 2016.
\bibitem [K17]{k17}  Kim, J.,  Some remarks on Fourier restriction estimates. preprint
    2017. arXiv:1702.01231
\bibitem [L05] {lee05}  Lee, S.,  Bilinear restriction estimates for surfaces with
    curvatures of different signs,  Trans. Amer. Math. Soc., Vol.
    358, No. 8, 3511--2533, 2005.
\bibitem [LV10] {lv10} Lee, S., Vargas, A.,  Restriction estimates for some surfaces with
    vanishing curvatures. J. Funct. Anal.  258  (2010),  no. 9, 2884--2909.
\bibitem [MVV96] {MVV1} Moyua, A., Vargas, A.,  Vega, L., Schr\"odinger maximal function
    and restriction
properties of the Fourier transform. Internat. Math. Res. Notices 16 (1996), 793--815.
\bibitem [MVV99] {MVV2} Moyua, A., Vargas, A.,  Vega, L., Restriction theorems and maximal
    operators related to oscillatory integrals in $\R^3$. Duke Math. J., 96 (3), (1999),
    547--574.
\bibitem [RdF83]{rdf} Rubio de Francia, J. L., Estimates for some square functions of
    Littlewood-Paley type. Publ. Sec. Mat. Univ. Aut\'onoma Barcelona  27  (1983),  no. 2,
    81--108.
\bibitem [St86] {St1}  Stein, E.M., Oscillatory Integrals in Fourier Analysis. Beijing
    Lectures in Harmonic Analysis. Princeton Univ. Press 1986.
\bibitem [Sto17] {Sto}  Stovall, B.,  Scale invariant Fourier restriction to a hyperbolic
    surface. Anal. PDE 12 (2019), no. 5, 1215--1224.
\bibitem [Str77] {Str} Strichartz, R. S., Restrictions of Fourier transforms to quadratic
    surfaces and decay of solutions of wave equations. Duke Math. J.  44  (1977), no. 3,
    705--714.
\bibitem [T01a] {T4} Tao, T.,  Endpoint bilinear restriction theorems for the cone, and
    some sharp null-form estimates. Math. Z. 238 (2001),215--268.
\bibitem [T03a] {T2} Tao, T., A Sharp bilinear restriction estimate for paraboloids.
    Geom. Funct. Anal. 13, 1359--1384, 2003.
\bibitem [To75] {To}  Tomas,  P. A., A restriction theorem for the Fourier transform.
    Bull. Amer. Math. Soc. 81 (1975), 477--478.
\bibitem [TVI00] {TV1} Tao, T.,  Vargas, A.,  A bilinear approach to cone multipliers I.
    Restriction estimates. Geom. Funct. Anal. 10, 185--215, 2000.
\bibitem [TVII00] {TV2} Tao, T.,  Vargas, A.,  A bilinear approach to cone multipliers II.
    Applications. Geom. Funct. Anal. 10, 216--258, 2000.
\bibitem [TVV98] {TVV} Tao, T., Vargas, A., Vega, L., A bilinear approach to the
    restriction and Kakeya conjectures. J. Amer. Math. Soc. 11 (1998) no. 4 , 967--1000.
\bibitem [V05]{v05}  Vargas, A.,  Restriction theorems for a surface with negative
    curvature. Math. Z. 249, 97--111 (2005).
\bibitem [W95] {W1}  Wolff, T., An improved bound for Kakeya type maximal functions.
    Rev. Mat. Iberoamericana 11 (1995), 651--674.
\bibitem [W01] {W2} Wolff, T.,  A Sharp Bilinear Cone Restriction Estimate. Ann. of Math., Second Series, Vol. 153, No. 3, 661--698, 2001.
\bibitem [Z74] {Z} Zygmund, A., On Fourier coefficients and transforms of functions of two
    variables. Studia Math. 50 (1974), 189--201.


\end{thebibliography}
\end{document}